\documentclass[12pt, reqno]{amsart}
\usepackage{amssymb, hyperref}
\usepackage{amsthm}
\usepackage{amsmath}
\usepackage{graphicx}
\usepackage[all,2cell]{xy}
\usepackage{verbatim}
\usepackage{tikz}
\usepackage{tikz-cd}
\usepackage{hyperref}
\usepackage{array}
\usepackage[capitalise]{cleveref}
\usepackage{epsfig,graphicx}
\usepackage{tikz}
\usepackage{tikz-cd}
\usepackage{float}
\usepackage{capt-of}
\tikzset{node distance=2cm, auto}
\usepackage{etoolbox}
\usetikzlibrary{graphs}
\usepackage{listings}
\usetikzlibrary{backgrounds}
\usetikzlibrary{decorations.markings}
\tikzstyle{vertex}=[circle, draw, inner sep=0pt, minimum size=6pt]
\usepackage{geometry}
\usepackage{mathrsfs}
\usepackage{caption}
\numberwithin{equation}{section}
\newtheorem*{theorem*}{Theorem}
\newtheorem*{corollary*}{\bf Corollary}
\newtheorem*{remark*}{\bf Remark}
\usepackage{lipsum}
\usetikzlibrary{matrix,arrows,decorations.pathmorphing}
\usepackage{lineno}
\newtheorem{theorem}{Theorem}[section]
\newtheorem{corollary}[theorem]{Corollary}

\newtheorem{lemma}[theorem]{Lemma}

\newtheorem{proposition}[theorem]{Proposition}

\newcommand{\eat}[1]{}

\begin{document}
	\author{S. Senthamarai Kannan$^{1}$ and Pinakinath Saha}
	
	\address{Chennai Mathematical Institute, Plot H1, SIPCOT IT Park, Siruseri, Kelambakkam,  603103, India.}
	
	\email{kannan@cmi.ac.in.}
	
	\address{Tata Inst. of Fundamental Research,
		Homi Bhabha Road, Colaba
		Mumbai 400005, INDIA.}
	\email{psaha@math.tifr.res.in.}
	\title{Minimal Parabolic subgroups and Automorphism groups of Schubert varieties-II}
	\begin{abstract} 	
		Let $G$ be a simple algebraic group of adjoint type over the field $\mathbb{C}$ of complex numbers, $B$ be a Borel subgroup of $G$ containing a maximal torus $T$ of $G.$ In this article, we show that $\alpha$ is a co-minuscule root if and only if for any parabolic subgroup $Q$ containing $B$ properly, there is no Schubert variety $X_{Q}(w)$ in $G/Q$ such that the minimal parabolic subgroup $P_{\alpha}$ of $G$ is the  connected component, containing the identity automorphism of the group of all algebraic automorphisms of $X_{Q}(w).$	
	\end{abstract}
	\maketitle
	\subjclass{Mathematics Subject Classification:}{~Primary 14M15; secondary 14M17}
	
	\footnotetext[1]{Corresponding author: S. SENTHAMARAI KANNAN (e-mail: kannan@cmi.ac.in).}
	\maketitle
	\section{introduction}
	We recall that if $X$ is a projective variety over $\mathbb{C}$, the connected component, containing the identity automorphism of the group of all algebraic automorphisms of $X$ is an algebraic group (see \cite[Theorem 3.7, p.17]{MO}). Let $G$ be a simple algebraic group of adjoint type over $\mathbb{C}.$ Let $T$  be a maximal torus  of $G,$  and let $R$ be the set of roots with respect to $T.$ Let $R^{+}\subset R$ be a set of positive roots. Let $B^{+}$ be the Borel subgroup of $G$ containing $T,$ corresponding to $R^{+}$. Let $B$ be the Borel subgroup of $G$ opposite to $B^+$ determined by $T$. Let $W=N_{G}(T)/T$ denote the Weyl group of $G$ with respect to $T.$ 
	For $w \in W$, let $X(w):=\overline{BwB/B}$ denote the Schubert variety in $G/B$ corresponding to $w$.
	In \cite{Dem aut}, M. Demazure studied the automorphism group of the homogeneous space $G/P,$ where $P$ is a parabolic subgroup of $G.$ He proved that if $(G, P)$ is not exceptional, then the connected component, containing identity automorphism of the group of all algebraic automorphisms of $G/P$ is $G$ (see \cite[Theorem 2, p.75]{Akh}). The Lie algebra of $G$ may be identified with the Lie algebra of global vector fields $H^0(G/P, T_{G/P}).$ Let $Aut^{0}(X(w))$ denote the connected component, containing the identity automorphism of the group of all algebraic automorphisms of $X(w).$ 
	Let $\alpha_{0}$ denote the highest root of $G$ with respect to $T$ and $B^{+}.$ For the left action of $G$ on $G/B$, let $P_{w}$ denote the stabilizer of $X(w)$ in $G.$
	In \cite[p.772,Theorem 4.2(2)]{Kan}, the first named author proved that if $G$ is simply-laced and $X(w)$ is smooth, then we have $P_{w}=Aut^{0}(X(w))$ if and only if $w^{-1}(\alpha_{0})<0.$ Therefore, it is a natural question to ask whether given any parabolic subgroup $P$ of $G$ containing $B$ properly, is there a Schubert variety $X(w)$ in $G/B$ such that $P=Aut^{0}(X(w))$ ? If $P=B,$ there is no such Schubert variety in $G/B.$ In \cite{KP1}, we gave an affirmative answer to this question. Also, we gave some partial results for Schubert varieties in partial flag varieties of type $A_{n}.$ Therefore, it is  a natural question to ask whether given a parabolic subgroup $P$ of $G$ containing $B$ properly, is there a parabolic subgroup $Q$ containing $B$ properly, such that $P=Aut^0(X_{Q}(w))$ for some Schubert variety $X_{Q}(w)$ in $G/Q$? In \cite{KS2}, we proved that when $G$ is simply-laced,
	a fundamental weight $\omega_{\alpha}$ is minuscule if
	and only if for any parabolic subgroup $Q$ containing $B$ properly, there is no Schubert variety
	$X_{Q}(w)$ in $G/Q$ such that $P_{\alpha}=Aut^{0}(X_{Q}(w)).$
	Note that when $G$ is simply-laced $\omega_{\alpha}$ is minuscule if and only if $\alpha$ is co-minuscule. Therefore, it is  a natural question to ask what is the generalization of the above result to arbitrary simple algebraic group of adjoint type. In this article, we prove the
	following:
	\begin{theorem}
		Let $G$ be a simple algebraic group of adjoint type over $\mathbb{C}.$ Then a simple root $\alpha$ is co-minuscule if and only if for any parabolic subgroup $Q$ containing $B$ properly, there is no Schubert variety $X_{Q}(w)$ in $G/Q$ such that $P_{\alpha}=Aut^0(X_{Q}(w)),$ where $P_{\alpha}$ is the minimal parabolic subgroup containing $B$ corresponding to $\alpha$ (see Theorem \ref{theorem 9.2}).
	\end{theorem}
	
	The organization of this article is as follows: In Section 2, we introduce some notation and preliminaries on algebraic groups and Lie algebras. In Section 3, we study some properties of co-minuscule root. In Section 4, we prove a crucial lemma on a non-co-minuscule root involving the highest root. In Section 5, when $G$ is of type $B_{n},$ we prove that for any non-co-minuscule simple root $\alpha$ there is a Schubert variety $X_{\alpha}$ in a partial flag variety such that $P_{\alpha}=Aut^0(X_{\alpha})$ (see Proposition \ref{prop5.8} for more precise statement).
	In Section 6, when $G$ is of type $C_{n},$ we prove that for any non-co-minuscule simple root $\alpha$ there is a Schubert variety $X_{\alpha}$ in a partial flag variety such that $P_{\alpha}=Aut^0(X_{\alpha})$ (see Proposition \ref{prop6.7} for more precise statement). In Section 7, when $G$ is of type $F_{4},$ we prove that for any simple root $\alpha$ there is a Schubert variety $X_{\alpha}$ in a partial flag variety such that $P_{\alpha}=Aut^0(X_{\alpha})$ (see Proposition \ref{prop 7.6} for more precise statement). In Section 8, when $G$ is of type $G_{2},$ we prove that for any simple root $\alpha$ there is a Schubert variety $X_{\alpha}$ in a partial flag variety such that $P_{\alpha}=Aut^0(X_{\alpha})$ (see Proposition \ref{prop 8.3} for more precise statement). In Section 9, we first prove that if  $\alpha$ is co-minuscule and there exists a parabolic subgroup $Q$ of $G$ containing $B,$ and a Schubert variety $X_{Q}(w)$ in $G/Q$ such that $P_{\alpha}=Aut^0(X_{Q}(w)),$ then we have $Q=B$ (see Proposition \ref{theorem 9.1}). Next we use \cite[Theorem 9.2]{KS2}, Proposition \ref{theorem 9.1}, and the results of the Sections 5--8 to prove the main theorem of this article (see Theorem \ref{theorem 9.2} for more precise statement).

	\section{Notation and Preliminaries}
	In this section, we set up some notation and preliminaries. We refer to \cite{Hum1}, \cite{Hum2}, \cite{Jan} for preliminaries in algebraic groups and Lie algebras.

	Let $G,B,T, R, R^{+},$ and $W,$ be as in the introduction.    
	Let $S = \{\alpha_1,\ldots,\alpha_n\}$ denote the set of simple roots in $R^{+}.$ Every $\beta \in R$ can be expressed uniquely as  $\sum\limits_{i=1}^{n}k_{i}\alpha_{i}$ with integral coefficients $k_{i}$ with either all non-negative or all non-positive. This allows us to define the {\bf height} of a root (relative to $S$) by ht$(\beta)=\sum\limits_{i=1}^{n}k_{i}.$ For $\beta=\sum\limits_{i=1}^{n}k_{i}\alpha_{i} \in R,$ we define {\bf support} of $\beta$ to be the set $\{\alpha_{i}: k_{i}\neq 0 \}.$
	The simple reflection in  $W$ corresponding to $\alpha_i$ is denoted
	by $s_{i}$. Then $(W, S)$ is a Coxeter group (see \cite[Theorem 29.4, p.180]{Hum2}). There is a natural length function $\ell$ defined on $W.$ Let $\mathfrak{g}$ be the Lie algebra of $G$. 
	Let $\mathfrak{h}\subset \mathfrak{g}$ be the Lie algebra of $T$ and  $\mathfrak{b}\subset \mathfrak{g}$ be the Lie algebra of $B$. Let $X(T)$ denote the group of all characters of $T$. 
	We have $X(T)\otimes\mathbb{R}=Hom_{\mathbb{R}}(\mathfrak{h}_{\mathbb{R}}, \mathbb{R})$, the dual of the real form of $\mathfrak{h}$. The positive definite 
	$W$-invariant form on $Hom_{\mathbb{R}}(\mathfrak{h}_{\mathbb{R}}, \mathbb{R})$ 
	induced by the Killing form of $\mathfrak{g}$ is denoted by $(~,~)$. 
	We use the notation $\left< ~,~ \right>$ to
	denote $\langle \mu, \alpha \rangle  = \frac{2(\mu,
		\alpha)}{(\alpha,\alpha)}$,  for every  $\mu\in X(T)\otimes \mathbb{R}$ and $\alpha\in R$. For a subset $J$ of $S,$ we denote by $W_{J}$ the subgroup of $W$ generated by $\{s_{\alpha}:\alpha \in J\}$. Let $W^{J}:=\{w\in W: w(\alpha)\in R^{+}~ for ~ all ~ \alpha \in J\}.$  For each $w\in W_{J},$ choose a representative element $n_{w}\in N_{G}(T).$ Let $N_{J}:=\{n_{w}: w\in W_{J}\}.$  Let $P_{J}:=BN_{J}B.$
	For a simple root $\alpha_{i}$, we denote by $P_{i}$ the minimal parabolic subgroup $P_{\alpha_{i}}$ of $G.$  Let  $\{\omega_{i}:1\le i\le n\}$ be the set of fundamental dominant weights corresponding to $\{\alpha_{i}: 1\le i\le n\}.$  
	
	We recall the following definition and facts (see \cite[p.119-120]{BL}):
	
	A fundamental weight $\omega$ is said to be  {\bf minuscule} if $\omega$ satisfies $\langle\omega, \beta \rangle \le 1$ for all $\beta \in R^{+}.$ 
	A simple root $\alpha$ is said to be  {\bf co-minuscule} if $\alpha$ occurs with coefficient $1$ in the expression of the highest root $\alpha_{0}.$
	
	If $G$ is simply-laced, then $\omega_{\alpha}$ is  minuscule if and only if $\alpha$ is co-minuscule (see \cite[Lemma 3.1]{KS2}). Here, is the complete list of co-minuscule roots in non-simply-laced root systems.
	\begin{center}
		\begin{tabular}{ |p{1.3cm}|p{3.3cm}|p{3.5cm}| }
			\hline
			\multicolumn{3}{|c|}{ Co-minuscule root in non-simply-laced root system} \\
			\hline
			no.&Root System& Co-minuscule root \\
			\hline
			$1.$ &	$B_{n}$ $(n\ge 2)$  & $\alpha_{1}$     \\

			2. & $C_{n}$ $(n\ge 3)$  & $\alpha_{n}$\\
			
			3.&	$F_{4}$  & none       \\
			
			4.&	$G_{2}$  &    none   \\
			\hline
		\end{tabular}
		
		\begin{center}
			\bf{Table 1: Co-minuscule roots}
		\end{center}
	\end{center}

	We recall the following proposition from \cite[Proposition 7.1, p.342-343]{CP}.
	
	Let $\alpha_{0}=\sum\limits_{i=1}^{n}c_{i}\alpha_{i},$ and $\check{\alpha_{0}}=\sum\limits_{i=1}^{n}\check{c_{i}}\check{\alpha_{i}}.$ We have $\check{\alpha_{0}}=\frac{2\alpha_{0}}{(\alpha_{0}, \alpha_{0})}=\frac{2}{(\alpha_{0}, \alpha_{0})}\sum\limits_{i=1}^{n}c_{i}\frac{(\alpha_{i}, \alpha_{i})}{2}\check{\alpha_{i}},$ hence $\check{c_{i}}=\frac{(\alpha_{i}, \alpha_{i})}{(\alpha_{0}, \alpha_{0})}c_{i}.$ The dual Coxeter number of $\mathfrak{g}$ is 
	$$g=1+\sum\limits_{i=1}^{n}\check{c_{i}}.$$
	
	Here, is the complete list of dual Coxeter number in non-simply-laced root system.
	\begin{center}
		\begin{tabular}{ |p{1.3cm}|p{3.3cm}|p{3.5cm}| }
			\hline
			\multicolumn{3}{|c|}{Dual Coxeter number in non-simply-laced root system} \\
			\hline
			no.&Root System& $g$ \\
			\hline
			$1.$ &	$B_{n}$ $(n\ge 2)$  & $2n-1$     \\

			2. & $C_{n}$ $(n\ge 3)$  & $n+1$\\
			
			3.&	$F_{4}$  & $9$       \\
			
			4.&	$G_{2}$  &    $4$  \\
			\hline
		\end{tabular}
		
		\begin{center}
			\bf{Table 2: Dual Coxeter number}
		\end{center}
	\end{center} 
	\begin{proposition}\label{Prop 2.1}
		Let $\alpha$ be any long root. Then we have 
		\begin{itemize}
			\item[(1)] There is a unique element $u_{\alpha}$ in $W$ of minimal length such that $u_{\alpha}^{-1}(\alpha_{0})=\alpha.$
		\end{itemize}
		
		\begin{itemize}
			\item[(2)] If $\alpha$ is in $S,$ then $\ell(u_{\alpha})=g-2.$
		\end{itemize}
	\end{proposition}
	\begin{proof}
		See \cite[Proposition 7.1, p.342-343]{CP}.
	\end{proof}
	\begin{corollary}\label{cor 2.2}
		Let $\alpha\in S$ be a long root, then there exists a unique element $v_{\alpha}$ of minimal length such that $v_{\alpha}^{-1}(\alpha_{0})=-\alpha.$ 
	\end{corollary}
	\begin{proof}
		Take $v_{\alpha}=u_{-\alpha}.$ 
	\end{proof}
	Note that $v_{\alpha}=s_{\alpha}u_{\alpha}$ for any long simple root $\alpha.$
	
	Now we discuss some preliminaries on the cohomology of vector bundles on Schubert varieties associated to the rational $B$-modules. Let $V$ be a rational $B$-module. Let $\phi:B\longrightarrow GL(V)$ be the corresponding homomorphism of algebraic groups. The total space of the vector bundle  $\mathcal{L}(V)$  on $G/B$ is defined by the set of equivalence classes 
	$\mathcal{L}(V)= G \times_{B} V$ corresponding to the following equivalence relation on $G\times V$:
	\begin{center}
		$(g,v)\sim (gb,\phi(b^{-1})\cdot v)$ for $g\in G, b\in B, v\in V.$ 
	\end{center}
	We denote by the restriction of $\mathcal{L}(V)$ to $X(w)$ also by $\mathcal{L}(V)$. We denote the cohomology modules $H^i(X(w), \mathcal{L}(V))$ by $H^i(w, V)$ ($i\in \mathbb{Z}_{\ge 0}$). If $V=\mathbb{C}_{\lambda}$ is one dimensional representation $\lambda: B\longrightarrow \mathbb{C}^{\times}$ of $B,$ then we denote $H^i(w, V)$ by $H^i(w, \lambda).$
	
	Let $L_{\alpha}$ denote the Levi subgroup of $P_{\alpha}$
	containing $T$. Note that $L_{\alpha}$ is the product of $T$ and the homomorphic image 
	$G_{\alpha}$ of $SL(2, \mathbb{C})$ via a homomorphism $\psi:SL(2, \mathbb{C})\longrightarrow L_{\alpha}$ (see  [7, II, 1.3]). We denote the intersection of $L_{\alpha}$ and $B$ by $B_{\alpha}$.  
	We note that the morphism $L_{\alpha}/B_{\alpha}\hookrightarrow P_{\alpha}/B$  induced by the inclusion $L_{\alpha}\hookrightarrow P_{\alpha}$ is an isomorphism. Therefore,  to compute the cohomology modules $H^{i}(P_{\alpha}/B, \mathcal{L}(V))$ ($0\leq i \leq 1$) for any $B$-module 
	$V,$ we treat $V$ as a $B_{\alpha}$-module  and we compute 
	$H^{i}(L_{\alpha}/B_{\alpha}, \mathcal{L}(V))$. 
	
	We use the following lemma to compute cohomology modules. The following lemma is due to Demazure (see \cite[p.271-272]{Dem}). Demazure used this lemma to prove  Borel-Weil-Bott's theorem.
	
	\begin{lemma}\label{lem 2.1}
		Let $w=\tau s_{\alpha},$ $\ell(w)=\ell(\tau) + 1,$ and $\lambda$ be a character of $B.$ Then we have 
		\begin{enumerate}
			
			\item [(1)] If $\langle \lambda , \alpha \rangle\ge0,$ then $H^{j}(w, \lambda)=H^{j}(\tau, H^0(s_{\alpha}, \lambda))$ for all $j\ge 0.$
			
			\item[(2)] If $\langle \lambda , \alpha \rangle\ge 0,$ then $H^j(w, \lambda)=H^{j+1}(w, s_{\alpha}\cdot \lambda)$ for all $j\ge0.$ 
			
			\item[(3)] If $\langle \lambda, \alpha\rangle\le -2 ,$ then $H^{j+1}(w, \lambda)=H^j(w, s_{\alpha}\cdot \lambda)$ for all $j\ge 0.$
			
			\item[(4)] If $\langle \lambda , \alpha \rangle=-1,$ then $H^j(w, \lambda)$ vanishes for every $j\ge0.$
		\end{enumerate}
	\end{lemma}
	
	Let $\pi:\hat{G}\longrightarrow G$ be the simply connected covering of $G$.  
	Let  $\hat{L_{\alpha}}$  (respectively, $\hat{B_{\alpha}}$)  be the inverse image 
	of $L_{\alpha}$ (respectively, of $B_{\alpha}$) in $\hat{G}$. Note that $\hat{L_{\alpha}}/\hat{B_{\alpha}}$ is isomorphic to $L_{\alpha}/B_{\alpha}$. We make use of this isomorphism to use the same notation for the vector bundle on $L_{\alpha}/B_{\alpha}$ associated to a $\hat{B_{\alpha}}$-module. Let $V$ be an irreducible  $\hat{L_{\alpha}}$-module and $\lambda$ be a character of $\hat{B_{\alpha}}$. 
	
	Then, we have
	
	\begin{lemma}\label{lem 3.1}
		\begin{enumerate}
			
			\item If $\langle \lambda , \alpha \rangle \geq 0$, then the $\hat{L_{\alpha}}$-module
			$H^{0}(L_{\alpha}/B_{\alpha} , V\otimes \mathbb{C}_{\lambda})$ 
			is isomorphic to the tensor product of  $ ~ V$ and 
			$H^{0}(L_{\alpha}/B_{\alpha} , \mathbb{C}_{\lambda})$. Further, we have  
			$H^{j}(L_{\alpha}/B_{\alpha} , V\otimes \mathbb{C}_{\lambda}) =0$ 
			for every $j\geq 1$.
			\item If $\langle \lambda , \alpha \rangle  \leq -2$, then we have  
			$H^{0}(L_{\alpha}/B_{\alpha} , V \otimes \mathbb{C}_{\lambda})=0.$ 
			Further, the $\hat{L_{\alpha}}$-module  $H^{1}(L_{\alpha}/B_{\alpha} , V \otimes \mathbb{C}_{\lambda})$ is isomorphic to the tensor product of  $~V$ and $H^{0}(L_{\alpha}/B_{\alpha} , 
			\mathbb{C}_{s_{\alpha}\cdot\lambda})$. 
			\item If $\langle \lambda , \alpha \rangle = -1$, then 
			$H^{j}( L_{\alpha}/B_{\alpha} , V \otimes \mathbb{C}_{\lambda}) =0$ 
			for every $j\geq 0$.
		\end{enumerate}
	\end{lemma}
	
	\begin{proof}
		
		By \cite[I, Proposition 4.8, p.53]{Jan} and \cite[I, Proposition 5.12, p.77]{Jan} for $j\ge0$,
		we have the following isomorphism as 
		$\hat{L_{\alpha}}$-modules:
		$$H^j(L_{\alpha}/B_{\alpha}, V \otimes \mathbb C_{\lambda})\simeq V \otimes
		H^j(L_{\alpha}/B_{\alpha}, \mathbb C_{\lambda}).$$ 
		Now the proof of the lemma follows from Lemma \ref{lem 2.1} by taking $w=s_{\alpha}$ 
		and the fact that $L_{\alpha}/B_{\alpha} \simeq P_{\alpha}/B$.
	\end{proof}
	
	We now state the following lemma on indecomposable 
	$\hat{B_{\alpha}}$ (respectively,  $B_{\alpha}$) modules which will be used in computing 
	the cohomology modules (see \cite [Corollary 9.1, p.30]{BKS}).
	
	\begin{lemma}\label{lem 3.2}
		\begin{enumerate}
			\item
			Any finite dimensional indecomposable $\hat{B_{\alpha}}$-module $V$ is isomorphic to 
			$V^{\prime}\otimes \mathbb{C}_{\lambda}$ for some irreducible representation
			$V^{\prime}$ of $\hat{L_{\alpha}}$, and some character $\lambda$ of $\hat{B_{\alpha}}$.
			
			\item
			Any finite dimensional indecomposable $B_{\alpha}$-module $V$ is isomorphic to 
			$V^{\prime}\otimes \mathbb{C}_{\lambda}$ for some irreducible representation
			$V^{\prime}$ of $\hat{L_{\alpha}}$, and some character $\lambda$ of $\hat{B_{\alpha}}$.
		\end{enumerate}
	\end{lemma}
	\begin{proof}
		Proof of part (1) follows from \cite [Corollary 9.1, p.30]{BKS}.
		
		Proof of part (2) follows from the fact that every $B_{\alpha}$-module can be viewed  as a 
		$\hat{B_{\alpha}}$-module via the natural homomorphism. 
	\end{proof}
	
	Here, we prove a Lemma that we use later. The following Lemma is independent of type of $G.$ 
	\begin{lemma}\label{lem 5.4} Let $v\in W.$
		If $H^0(v, \mathfrak{b})_{-\alpha_{i}}\neq 0,$ then $\mathbb{C}h(\alpha_{i})\subseteq H^0(v, \mathfrak{b}).$ Hence, the two dimensional $B_{\alpha_{i}}$-module $\mathbb{C}h(\alpha_{i})\oplus \mathbb{C}_{-\alpha_{i}}$ is a direct summand of $H^0(v, \mathfrak{b}).$
	\end{lemma}
	\begin{proof}
		We prove by induction on $\ell(v).$ Assume that $\ell(v)>0.$ Choose a simple root $\alpha$  such that $s_{\alpha}v<v.$ Then by SES $H^0(v,\mathfrak{b})_{-\alpha_{i}}=H^0(s_{\alpha}, H^0(s_{\alpha}v,\mathfrak{b}))_{-\alpha_{i}}\neq 0.$ Then there exists an indecomposable $\hat{B}_{\alpha}$-summand $V$ of $H^0(s_{\alpha}v,\mathfrak{b})$ such that $H^0(s_{\alpha}, V)_{-\alpha_{i}}\neq 0.$ Note that $\mathbb{C}h(\alpha_{i})\oplus \mathbb{C}_{-\alpha_{i}}$ is a $B_{\alpha_{i}}$-direct summand of $\mathfrak{b}$ we have $\alpha\neq \alpha_{i}.$ Therefore, we have  $\langle -\alpha_{i}, \alpha\rangle \ge 0.$ Since  $H^0(v,\mathfrak{b})\subseteq H^0(s_{\alpha}v, \mathfrak{b}),$ we have $H^0(s_{\alpha}v, \mathfrak{b})_{-\alpha_{i}}\neq 0.$ Therefore, by induction hypothesis $\mathbb{C}h(\alpha_{i})\subseteq H^0(s_{\alpha}v,\mathfrak{b}).$ Thus $\mathbb{C}h(\alpha_{i})$ is one dimensional $\hat{L}_{\alpha}$-submodule of $H^0(s_{\alpha}v,\mathfrak{b}).$ Hence, by Lemma \ref{lem 3.1}(1), $\mathbb{C}h(\alpha_{i})\subseteq H^0(v,\mathfrak{b}).$ Therefore, the statement follows.
		
	\end{proof}

	\section{Some properties of co-minuscule root}
	Now onwards till the end of section 8, we will assume that $G$ is non-simply-laced. In this section, we prove some properties of co-minuscule simple roots.
	\begin{lemma}\label{Lemma 1.1}
		Every co-minuscule root is a long root.
	\end{lemma}
	\begin{proof}
		Assume that $\alpha$ is a co-minuscule root. If possible suppose that $\alpha$ is a short root. Then there is a smallest positive integer $r$ such that $\alpha, \alpha_{i_{1}},\alpha_{i_{2}},...,\alpha_{i_{r}}$ are consecutive nodes with $\alpha_{i_{1}},\alpha_{i_{2}},...,\alpha_{i_{r-1}}$ are short roots and $\alpha_{i_{r}}$ is a long root. Since $\alpha_{i_{r}}$ is a long root, $\langle  \alpha_{i_{r}},\alpha_{i_{r-1}} \rangle \le -2.$ Let $\beta=s_{i_{1}}\cdots s_{i_{r-1}}(\alpha_{i_{r}}).$ Then $\beta =s_{i_{1}}\cdots s_{i_{r-1}}(\alpha_{i_{r}})=\alpha_{i_{r}}+ a(\sum\limits_{j=i}^{r-1}\alpha_{i_{j}})$ where $a\ge 2.$ Since $\alpha,$ $\alpha_{i_{1}},\ldots ,\alpha_{i_{r-1}}$ are short roots such that $\alpha,\alpha_{i_{1}},\alpha_{i_{2}},...,\alpha_{i_{r}}$ are consecutive nodes, we have $\langle \alpha_{i_{1}}, \alpha \rangle =-1,$ and $\langle \alpha_{i_{j}}, \alpha\rangle=0$ for all $2\le j\le r.$ Therefore, $\beta$ is a positive root such that $\langle \beta, \alpha\rangle=-a\le -2.$ Thus $s_{\alpha}(\beta)$ is a positive root such that coefficient of $\alpha$ in $s_{\alpha}(\beta)$ is at least 2, which is a contradiction to the fact that $\alpha$ is co-minuscule simple root. 
	\end{proof}
	\begin{lemma}\label{lem 2.3}
		$\alpha_{r}$ is  co-minuscule if and only if $w_{0,S \setminus \{\alpha_{r}\}}(\alpha_{r})=\alpha_{0}.$ 
	\end{lemma}
	\begin{proof}
		Assume that $\alpha_{r}$ is co-minuscule, i.e., the coefficient of $\alpha_{r}$ in the expression of $\alpha_{0}$ is $1.$  Note that $-\alpha_{r}$ is $L_{S\setminus\{\alpha_{r}\}}$ dominant. Therefore, $w_{0, S\setminus\{\alpha_{r}\}}(-\alpha_{r})$ is $L_{S\setminus\{\alpha_{r}\}}$ negative dominant. Further, the coefficient of $\alpha_{r}$ in $w_{0,S\setminus\{\alpha_{r}\}}(-\alpha_{r})$ is $-1.$ On the other hand, if $\langle w_{0, S\setminus\{\alpha_{r}\} }(-\alpha_{r}), \alpha_{r}\rangle\ge 1,$ then the coefficient of $\alpha_{r}$ in the expression of $s_{r}(w_{0,S\setminus\{\alpha_{r}\}}(-\alpha_{r}))$ is $\le- 2.$ Since $-\alpha_{0}\le s_{r}w_{0, S\setminus \{\alpha_{r}\}}(-\alpha_{r}),$ the coefficient of $\alpha_{r}$ in the expression of $-\alpha_{0}$ is $\le -2.$ This is a contradiction to the hypothesis that $\alpha_{r}$ is co-minuscule.  Hence, $w_{0, S\setminus\{\alpha_{r}\}}(-\alpha_{r})$ is negative dominant. Further, by Lemma \ref{Lemma 1.1}, $w_{0, S\setminus \{\alpha_{r}\}}(\alpha_{r})$ is a long root. Therefore, we have  $w_{0, S\setminus\{\alpha_{r}\}}(\alpha_{r})=\alpha_{0}.$
		
		Conversely, since $w_{0,S\setminus\{\alpha_{r}\}}(\alpha_{r})=\alpha_{0},$ the coefficient of $\alpha_{r}$ in the expression of $\alpha_{0}$ is $1.$ Therefore, $\alpha_{r}$ is co-minuscule.
	\end{proof}
	Since $G$ is non-simply-laced, the Dynkin diagram automorphism induced by $-w_{0}$ is the identity automorphism of the Dynkin diagram (see \cite[p.216, p.217, p.233]{Bou}).
	
	\begin{lemma} \label{lemma 1.2}
		Assume that $\alpha_{r}$ is co-minuscule.
		Let $v\in W^{S\setminus\{\alpha_{r}\}}.$ Then we have $v={w_{0}^{S\setminus\{\alpha_{r}\}}}$ if and only if $v(\alpha_{0})<0.$ 
	\end{lemma}
	\begin{proof}
		Assume that $v=w_{0}^{S\setminus\{\alpha_{r}\}}.$ Since  $w_{0}=w_{0}^{S\setminus\{\alpha_{r}\}}w_{0, S\setminus \{ \alpha_{r}\}},$  we have $w_{0}^{S\setminus\{\alpha_{r}\}}=w_{0} w_{0,S\setminus\{\alpha_{r}\}}.$ Therefore, by using Lemma \ref{lem 2.3}, we have $w_{0}^{S\setminus\{\alpha_{r}\}}(\alpha_{0})=w_{0}(\alpha_{r})=-\alpha_{r}.$ 
		
		Conversely, assume that $v(\alpha_{0})<0.$ Then by using Lemma \ref{lem 2.3}, we have $s_{r}{w_{0}}^{S\setminus\{\alpha_{r}\}}(\alpha_{0})=\alpha_{r}.$ Therefore, we have $v \nleq s_{r}w_{0}^{S\setminus\{\alpha_{r}\}}.$ Otherwise $v(\alpha_{0})\ge s_{r}w_{0}^{S\setminus\{\alpha_{r}\}}(\alpha_{0})=\alpha_{r}.$ Therefore, $v(\alpha_{0})$ is a positive root, which is a contradiction to the hypothesis that $v(\alpha_{0})<0.$ Since $v,s_{r}w_{0}^{S\setminus\{\alpha_{r}\}}\in W^{S\setminus\{\alpha_{r}\}},$ we have $v=w_{0}^{S\setminus\{\alpha_{r}\}}.$  
	\end{proof}

	\section{preliminaries on non co-minuscule root}
	In this section, we prove a crucial lemma for a non co-minuscule root  of type $B,$ $C,$ $F_{4},$ or $G_{2}$ associated to $\alpha_{0}.$ We recall the Dynkin diagram of $B_{n}, C_{n}, F_{4}, G_{2}$ (see \cite[Theorem 11.4, p.57-58]{Hum1}):

	\vspace{.2cm}
	\setlength{\unitlength}{1.2cm}
	\begin{picture}(10,0)
		\thicklines
		\put(.2,0){\circle*{0.2}}
		\put(0,-0.5){$\alpha_{1}$}
		\put(1.2,0){\circle*{0.2}}
		\put(.2,0){\line(1,0){1}}
		\put(1.2,0){\line(1,0){1}}
		\put(1,-0.5){$\alpha_{2}$}
		\put(2.2,0){\circle*{0.2}}
		\put(2,-0.5){$\alpha_{3}$}
		\put(1.6,0){\line(1,0){1}}
		\put(7.2,0){\circle*{0.2}}
		\put(7,-0.5){$\alpha_{n-2}$}
		\put(8.2,0){\circle*{0.2}}
		\put(7.2,0){\line(1,0){1}}
		\put(8.2,0.1){\line(1,0){1}}
		\put(8.2,-0.1){\line(1,0){1}}
		\put(8.2,.1){\line(1,0){1}}
		\put(8,-0.5){$\alpha_{n-1}$}
		\put(8.5,.3){\line(1,-1){.3}}
		\put(8.5,-.3){\line(1,1){.3}}
		\put(9.2,0){\circle*{0.2}}
		\put(9,-.5){$\alpha_{n}$}
		\put(4,0){\circle*{0.2}}
		\put(3.9,-0.5){$\alpha_{i}$}
		\put(1.7,-1.9){Figure 1: Dynkin diagram of $B_{n}(n\ge 2).$}
		\put(1.7,-2.4){$\alpha_{0}=\alpha_{1}+2(\alpha_{2}+\alpha_{3}+\cdots +\alpha_{n}).$}
		\put(4,0){\line(1,0){1}}
		\put(5,0){\circle*{0.2}}
		\put(4.9,-0.5){$\alpha_{i+1}$}
		\put(4,0){\circle*{0.2}}
		\put(4,0){\line(1,0){1}}
		\put(5,0){\circle*{0.2}}
		\put(3.5,0){\line(1,0){1}}
		\put(4.5,0){\line(1,0){1}}
		\put(6.7,0){\line(1,0){1}}
	\end{picture}
	
	\vspace{3cm}
	\setlength{\unitlength}{1.2cm}
	\begin{picture}(10,0)
		\thicklines
		\put(.2,0){\circle*{0.2}}
		\put(0,-0.5){$\alpha_{1}$}
		\put(1.2,0){\circle*{0.2}}
		\put(.2,0){\line(1,0){1}}
		\put(1.2,0){\line(1,0){1}}
		\put(1,-0.5){$\alpha_{2}$}
		\put(2.2,0){\circle*{0.2}}
		\put(2,-0.5){$\alpha_{3}$}
		\put(1.6,0){\line(1,0){1}}
		\put(7.2,0){\circle*{0.2}}
		\put(7,-0.5){$\alpha_{n-2}$}
		\put(8.2,0){\circle*{0.2}}
		\put(7.2,0){\line(1,0){1}}
		\put(8.2,0.1){\line(1,0){1}}
		\put(8.2,-0.1){\line(1,0){1}}
		\put(8.2,.1){\line(1,0){1}}
		\put(8,-0.5){$\alpha_{n-1}$}
		\put(9.2,0){\circle*{0.2}}
		\put(8.6,0){\line(1,1){.3}}
		\put(8.6,0){\line(1,-1){.3}}
		\put(9,-.5){$\alpha_{n}$}
		\put(4,0){\circle*{0.2}}
		\put(3.9,-0.5){$\alpha_{i}$}
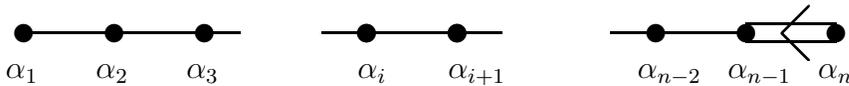
		\put(1.7,-1.9){Figure 2: Dynkin diagram of $C_{n}(n\ge 3).$}
		\put(1.7,-2.4){$\alpha_{0}=2(\alpha_{1}+\alpha_{2}+\cdots +\alpha_{n-1})+\alpha_{n}=2\omega_{1}.$}
		\put(4,0){\line(1,0){1}}
		\put(5,0){\circle*{0.2}}
		\put(4.9,-0.5){$\alpha_{i+1}$}
		\put(4,0){\circle*{0.2}}
		\put(4,0){\line(1,0){1}}
		\put(5,0){\circle*{0.2}}
		\put(3.5,0){\line(1,0){1}}
		\put(4.5,0){\line(1,0){1}}
		\put(6.7,0){\line(1,0){1}}
	\end{picture}
	
	\vspace{3cm}
	\hspace{3.5cm}
	\setlength{\unitlength}{1.4cm}
	\begin{picture}(10,0)
		\thicklines
		\put(.2,0){\circle*{0.2}}
		\put(0,-0.5){$\alpha_{1}$}
		\put(1.2,0){\circle*{0.2}}
		\put(.2,0){\line(1,0){1}}
		\put(1.2,-0.1){\line(1,0){1}}
		\put(1,-0.5){$\alpha_{2}$}
		\put(1.2,0.1){\line(1,0){1}}
		\put(1.5,0.2){\line(2,-1){.4}}
		\put(1.5,-0.2){\line(2,1){.4}}
		\put(2.2,0){\circle*{0.2}}
		\put(2,-0.5){$\alpha_{3}$}
		\put(2.2,0){\line(1,0){1}}
		\put(3.2,0){\circle*{0.2}}
		\put(3.1,-0.5){$\alpha_{4}$}
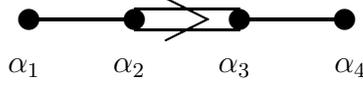
		\put(0,-1.5){Figure 3: Dynkin diagram of $F_{4}.$}
		\put(0,-1.9){$\alpha_{0}=2\alpha_{1}+3\alpha_{2}+4\alpha_{3}+2\alpha_{4}=\omega_{1}.$}
	\end{picture}
	
	\vspace{4cm}
	\hspace{4.5cm}
	\setlength{\unitlength}{1.4cm}
	\begin{picture}(10,0)
		\thicklines
		\put(.2,0){\circle*{0.3}}
		\put(0,-0.5){$\alpha_{2}$}
		\put(1.2,0){\circle*{0.3}}
		\put(.2,.1){\line(1,0){1}}
		\put(.5,-.3){\line(1,1){.3}}
		\put(.5,.3){\line(1,-1){.3}}
		\put(.2,0){\line(1,0){1}}
		\put(.2,-.1){\line(1,0){1}}
		\put(1,-0.5){$\alpha_{1}$}
		\put(-1.4,-1.5){Figure 4: Dynkin diagram of $G_{2}.$}
		\put(-.1,-1.9){$\alpha_{0}=3\alpha_{1}+2\alpha_{2}=\omega_{2}.$}
	\end{picture}
	\vspace{2.1cm}

	We now prove
	\begin{lemma}\label{Lemma 2.1}
		Assume that $G$ is of type $B,$ $C,$ $F_{4},$ or $G_{2}.$ Then we have $w_{0, S\setminus \{\alpha_{r}\}}(\alpha_{r})=\alpha_{0}-\alpha_{r},$ where $\alpha_{r}$ is the simple root such that $\alpha_{0}=\omega_{r}$ except for type $C$ and $B_{2}$ in these cases $\alpha_{0}=2\omega_{1}$ (respectively, $\alpha_{0}=2\omega_{2}$) and $\alpha_{r}=\alpha_{1}$ (respectively, $\alpha_{r}=\alpha_{2}$).
	\end{lemma}
	\begin{proof}
		Note that $-\alpha_{r}$ is $L_{S\setminus\{\alpha_{r}\}}$-dominant. Therefore, $w_{0, S\setminus\{\alpha_{r}\}}(-\alpha_{r})$ is $L_{S\setminus\{\alpha_{r}\}}$ negative dominant.
		
		{\bf Case I:} $G$ is of type $B.$
		
		First note that in type $B_{2},$ $\alpha_{0}=\alpha_{1}+2\alpha_{2}=2\omega_{2}.$ Thus, $w_{0, S\setminus\{\alpha_{2}\}}(\alpha_{2})=\alpha_{0}-\alpha_{2}.$  In type $B_{3},$ we have $\alpha_{0}=\omega_{2}=\alpha_{1}+2\alpha_{2}+2\alpha_{3}.$ Therefore, $w_{0,S\setminus\{\alpha_{2}\}}(\alpha_{2})=s_{1}s_{3}(\alpha_{2})=\alpha_{1}+\alpha_{2}+2\alpha_{3}=\alpha_{0}-\alpha_{2}.$

		In type $B_{n}(n\ge 4),$ we have $\alpha_{0}=\omega_{2}.$
		Then we claim that $\langle w_{0, S\setminus\{\alpha_{2}\}}(-\alpha_{2}), \alpha_{2}\rangle\ge 1.$ Assume on the contrary that $\langle w_{0, S\setminus\{\alpha_{2}\}}(-\alpha_{2}), \alpha_{2}\rangle\le 0,$ then $w_{0, S\setminus\{\alpha_{2}\}}(-\alpha_{2})$ is negative dominant. Further, since $\alpha_{2}$ is a long root, $w_{0, S\setminus\{\alpha_{2}\}}(-\alpha_{2})=-\alpha_{0}.$ Therefore, by Lemma \ref{lem 2.3}, $\alpha_{2}$ is co-minuscule which is a contradiction. Since $\langle \alpha_{i},\alpha_{2}\rangle=0$ for $i\neq 1,2,3,$ and $w_{0,S\setminus\{\alpha_{2}\}}(-\alpha_{2})$ is $L_{S\setminus\{\alpha_{2}\}}$-negative dominant, we have $\langle s_{2}w_{0, S\setminus\{\alpha_{2}\}}(-\alpha_{2}), \alpha_{i}\rangle\le 0$ for $i\neq 1,3.$ Further, we have $\langle s_{2}w_{0, S\setminus\{\alpha_{2}\}}(-\alpha_{2}), \alpha_{i}\rangle=\langle w_{0, S\setminus\{\alpha_{2}\}}(-\alpha_{2}), \alpha_{2}+\alpha_{i}\rangle$ for $i=1,3.$ Since $\alpha_{2}$ is a long root, by the above discussion we have $\langle w_{0, S\setminus\{\alpha_{2}\}}(-\alpha_{2}), \alpha_{2}\rangle=1.$ Moreover, since  $w_{0,S\setminus \{\alpha_{2}\}}(\alpha_{i})=-\alpha_{i}$ for $i=1,3,$ and $\alpha_{1},\alpha_{2},\alpha_{3}$ are long root, we have 
		$\langle w_{0,S\setminus\{\alpha_{2}\}}(-\alpha_{2}), \alpha_{i}+\alpha_{2}\rangle=\langle w_{0,S\setminus\{\alpha_{2}\}}(-\alpha_{2}), \alpha_{i}\rangle +\langle w_{0,S\setminus\{\alpha_{2}\}}(-\alpha_{2}), \alpha_{2}\rangle=0$ for $i=1,3.$ Thus $s_{2}w_{0,S\setminus\{\alpha_{2}\}}(-\alpha_{2})$ is a negative dominant. Since $\alpha_{2}$ is a long root, we have $s_{2}w_{0,S\setminus\{\alpha_{2}\}}(-\alpha_{2})=-\alpha_{0}.$ So, we have $w_{0,S\setminus\{\alpha_{2}\}}(\alpha_{2})=\alpha_{0}-\alpha_{2}.$

		{\bf Case II:} $G$ is of type $C.$

		In type $C,$ we have $\alpha_{0}=2\omega_{1}.$ We claim that $\langle w_{0, S\setminus\{\alpha_{1}\}}(-\alpha_{1}), \alpha_{1}\rangle \leq 0.$ Assume
		on the contrary that $\langle w_{0, S\setminus\{\alpha_{1}\}}(-\alpha_{1}), \alpha_{1}\rangle\ge1.$
		Since $\langle \alpha_{i},\alpha_{1}\rangle=0$ for $i\neq 1,2$ and $w_{0,S\setminus\{\alpha_{1}\}}(-\alpha_{1})$ is $L_{S\setminus\{\alpha_{1}\}}$-negative dominant, we have $\langle s_{1}w_{0, S\setminus\{\alpha_{1}\}}(-\alpha_{1}), \alpha_{i}\rangle\le 0$ for $i\neq 2.$ Note that $\langle s_{1}w_{0, S\setminus\{\alpha_{1}\}}(-\alpha_{1}), \alpha_{2}\rangle$ $=\langle w_{0, S\setminus\{\alpha_{1}\}}(-\alpha_{1}), \alpha_{1}+\alpha_{2}\rangle.$ Since $\alpha_{1}, \alpha_{2},\alpha_{1}+\alpha_{2}$ are short roots, $\langle w_{0, S\setminus\{\alpha_{1}\}}(-\alpha_{1}),\alpha_{1}+ \alpha_{2}\rangle=\langle w_{0, S\setminus\{\alpha_{1}\}}(-\alpha_{1}), \alpha_{1}\rangle+\langle w_{0, S\setminus\{\alpha_{1}\}}(-\alpha_{1}), \alpha_{2}\rangle.$ Further, since $\alpha_{1}$ is a short root and $\langle w_{0, S\setminus\{\alpha_{1}\}}(-\alpha_{1}), \alpha_{1}\rangle\ge1,$ we have $\langle w_{0, S\setminus\{\alpha_{1}\}}(-\alpha_{1}), \alpha_{1}\rangle=1.$ Since  $w_{0,S\setminus \{\alpha_{1}\}}(\alpha_{2})=-\alpha_{2},$ $\langle w_{0,S\setminus\{\alpha_{1}\}}(-\alpha_{1}), \alpha_{1}+ \alpha_{2}\rangle=0.$ Thus $s_{1}w_{0,S\setminus\{\alpha_{1}\}}(-\alpha_{1})$ is a negative dominant. Since $\alpha_{1}$ is a short root, we have $s_{1}w_{0,S\setminus\{\alpha_{1}\}}(-\alpha_{1})=-\beta_{0},$ where $\beta_{0}$ is the highest short root. Note that $\beta_{0}=\alpha_{0}-\alpha_{1}.$ Thus we have $w_{0,S\setminus\{\alpha_{1}\}}(\alpha_{1})$ $=\alpha_{0}-\alpha_{1}.$ This is a contradiction to the hypothesis that $\langle w_{0, S\setminus\{\alpha_{1}\}}(-\alpha_{1}), \alpha_{1}\rangle \ge 1.$ Thus we have $\langle w_{0, S\setminus\{\alpha_{1}\}}(-\alpha_{1}), \alpha_{1}\rangle\le 0.$ Then $w_{0, S\setminus\{\alpha_{1}\}}(-\alpha_{1})$ is negative dominant. Further, since $\alpha_{1}$ is a short root, $w_{0, S\setminus\{\alpha_{1}\}}(\alpha_{1})=\beta_{0}=\alpha_{0}-\alpha_{1}.$ 
		
		{\bf Case III:} $G$ is of type $F_{4}.$

		In type $F_{4},$ we have $\alpha_{0}=2\alpha_{1}+3\alpha_{2}+4\alpha_{3}+2\alpha_{4}=\omega_{1}.$ We claim that $\langle w_{0, S\setminus\{\alpha_{1}\}}(-\alpha_{1}), \alpha_{1}\rangle\ge 1.$ Assume on the contrary that $\langle w_{0, S\setminus\{\alpha_{1}\}}(-\alpha_{1}), \alpha_{1}\rangle\le 0,$ then $w_{0, S\setminus\{\alpha_{1}\}}(-\alpha_{1})$ is negative dominant. Further, since $\alpha_{1}$ is a long root, $w_{0, S\setminus\{\alpha_{1}\}}(-\alpha_{1})=-\alpha_{0}.$ Therefore, by Lemma \ref{lem 2.3}, $\alpha_{1}$ is co-minuscule which is a contradiction. Since $\langle \alpha_{i},\alpha_{1}\rangle=0$ for $i\neq 1,2$ and $w_{0,S\setminus\{\alpha_{1}\}}(-\alpha_{1})$ is $L_{S\setminus\{\alpha_{1}\}}$-negative dominant, we have $\langle s_{1}w_{0, S\setminus\{\alpha_{1}\}}(-\alpha_{1}), \alpha_{i}\rangle\le 0$ for $i\neq 2.$ Further, $\langle s_{1}w_{0, S\setminus\{\alpha_{1}\}}(-\alpha_{1}), \alpha_{2}\rangle=\langle w_{0, S\setminus\{\alpha_{1}\}}(-\alpha_{1}), \alpha_{1}+\alpha_{2}\rangle.$ Since $\alpha_{1},\alpha_{2},\alpha_{1}+\alpha_{2}$ are long roots, we have $\langle s_{1}w_{0,S\setminus\{\alpha_{1}\}}(-\alpha_{1}), \alpha_{2}\rangle=\langle w_{0,S\setminus\{\alpha_{1}\}}(-\alpha_{1}), \alpha_{1}\rangle+\langle w_{0,S\setminus\{\alpha_{1}\}}(-\alpha_{1}), \alpha_{2}\rangle.$ Moreover, since $\alpha_{1}$ is a long root and  $\langle w_{0, S\setminus\{\alpha_{1}\}}(-\alpha_{1}), \alpha_{1}\rangle\ge 1,$ we have $\langle w_{0,S\setminus\{\alpha_{1}\}}(-\alpha_{1}), \alpha_{1}\rangle=1.$ Further, we have $\langle w_{0,S\setminus\{\alpha_{1}\}}(-\alpha_{1}), \alpha_{2}\rangle=\langle -\alpha_{1}, w_{0,S\setminus\{\alpha_{1}\}}(\alpha_{2})\rangle=\langle -\alpha_{1},-\alpha_{2}\rangle=-1.$ Therefore, we have $\langle s_{1}w_{0,S\setminus\{\alpha_{1}\}}(-\alpha_{1}), \alpha_{2}\rangle=0.$ Thus $s_{1}w_{0,S\setminus\{\alpha_{1}\}}(-\alpha_{1})$ is negative dominant. Therefore, since $\alpha_{1}$ is a long root, we have $s_{1}w_{0,S\setminus\{\alpha_{1}\}}(-\alpha_{1})=-\alpha_{0}.$ So, we have $w_{0,S\setminus\{\alpha_{1}\}}(\alpha_{1})=\alpha_{0}-\alpha_{1}.$
		
		{\bf Case IV:} $G$ is of type $G_{2}.$
		
		In type $G_{2},$ we have $\alpha_{0}=3\alpha_{1}+2\alpha_{2}=\omega_{2}.$ Then we have $w_{0, S\setminus \{\alpha_{2}\}}=s_{1}.$ Therefore, $w_{0,S\setminus \{\alpha_{2}\}}(\alpha_{2})=\alpha_{2}+3\alpha_{1}=\alpha_{0}-\alpha_{2}.$
		
	\end{proof}

	\section{ $G$ is of type $B_{n}$  $(n\ge 2)$}
	In this section, we assume that $G$ is of type $B_{n}.$ Further, we prove that for any $2\le i\le n,$ there exists a Schubert variety $X_{P_{n}}(w_{i})$ in $G/P_{n}$ such that $P_{i}=Aut^0(X_{P_{n}}(w_{i})).$ 
	
	Since $G$ is of type $B_{n} (n\ge 2),$ we have $\alpha_{0}=\alpha_{1}+2(\alpha_{2}+\alpha_{3}+\cdots +\alpha_{n-1}+\alpha_{n}).$ Recall that by Corollary \ref{cor 2.2}, there exists a unique element $v_{n-1}$ in $W$ of minimal length such that $v_{n-1}^{-1}(\alpha_{0})=-\alpha_{n-1}.$  
	\begin{lemma}\label{lem5.1}
		Then $v_{n-1}$ satisfies the following:
		\begin{itemize}
			\item [(i)] $v_{n-1}=(s_{2}s_{3}\cdots s_{n-1}s_{n})(s_{1}s_{2}s_{3}\cdots s_{n-1}).$
			
			\item[(ii)] If $n=2,$ then we have $v_{n-1}^{-1}(\alpha_{1})=\alpha_{1}+2\alpha_{2}$ and $v_{n-1}^{-1}(\alpha_{2})=-(\alpha_{1}+\alpha_{2}).$
			
			\item[(iii)] If $n\ge 3,$ then we have $v_{n-1}^{-1}(\alpha_{1})=\alpha_{n-1}+2\alpha_{n},$ $v_{n-1}^{-1}(\alpha_{2})=-(\alpha_{1}+\cdots +\alpha_{n-2}+2\alpha_{n-1}+2\alpha_{n}),$ $v_{n-1}^{-1}(\alpha_{j})=\alpha_{j-2}$ for $3\le j\le n-1,$ $v_{n-1}^{-1}(\alpha_{n})=\alpha_{n-2}+\alpha_{n-1}+\alpha_{n}.$
		\end{itemize}
	\end{lemma}
	\begin{proof}
		Proof of (i): Let ${v}_{n-1}'=(s_{2}s_{3}\cdots s_{n-1}s_{n})(s_{1}s_{2}s_{3}\cdots s_{n-1}).$ Then by the usual calculation we have ${{v}_{n-1}'}^{-1}(\alpha_{0})=-\alpha_{n-1}.$ Moreover, since $\ell({v}_{n-1}')=\ell(v_{n-1}),$ by Corollary \ref{cor 2.2}, we have $v_{n-1}={v}_{n-1}'.$
		
		Proofs of (ii) and (iii) follow from the usual calculation.
	\end{proof}
	
	Let $x_{i}=w_{0,S\setminus\{\alpha_{2},\alpha_{i}\}}w_{0,S\setminus\{\alpha_{2}\}}=((w_{0,S\setminus\{\alpha_{2}\}})^{S\setminus\{\alpha_{2}, \alpha_{i}\}})^{-1}$ and $w_{i}=x_{i}v_{n-1}$ for all $2\le i\le n.$ We observe that $R^{+}(v_{n-1}^{-1})\subseteq\{\beta\in R^{+}:  \alpha_{2}\le \beta  \}.$ On the other hand, $R^{+}(x_{i})\subseteq \mathbb{Z}_{\ge 0}(S\setminus\{\alpha_{2}\})\cap R^{+}$ for all $2\le i\le n.$ Therefore, $R^{+}(v_{n-1}^{-1})\cap R^{+}(x_{i})=\emptyset.$ Hence, $\ell(w_{i})=\ell(x_{i})+\ell(v_{n-1})$ for all $2\le i\le n.$ 
	\begin{lemma}\label{lem5.2}
		We have
		\begin{itemize}
			\item[(i)] $v_{n-1}(\alpha_{n})=\alpha_{1}+\cdots +\alpha_{n}.$
			
			\item[(ii)] $w_{0, S\setminus\{\alpha_{2}\}}v_{n-1}(\alpha_{n})=\alpha_{2}+\cdots +\alpha_{n}.$
			
			\item[(iii)] $w_{i}(\alpha_{n})$ is a non-simple positive root for all $2\le i\le n.$
		\end{itemize}
	\end{lemma}
	\begin{proof}
		Proof of (i): follows from the usual calculation.
		
		Proof of (ii): By (i), we have $v_{n-1}(\alpha_{n})=\alpha_{1}+\cdots +\alpha_{n}.$ Note that $w_{0, S\setminus\{\alpha_{2}\}}(\alpha_{j})=-\alpha_{j}$ for all $ j\neq 2.$ On the other hand, by Lemma \ref{Lemma 2.1}, we have $w_{0,S\setminus\{\alpha_{2}\}}(\alpha_{2})=\alpha_{0}-\alpha_{2}.$Therefore, we have $w_{0,S\setminus\{\alpha_{2}\}}v_{n-1}(\alpha_{n})=\alpha_{2}+\cdots +\alpha_{n}.$
		
		Proof of (iii): For $i=2,$ the statement follows from (i). 
		
		By (ii), we have $w_{0, S\setminus\{\alpha_{2}\}}v_{n-1}(\alpha_{n})=\alpha_{2}+\cdots +\alpha_{n}.$ Since supp($w_{0, S\setminus\{\alpha_{2}\}}v_{n-1}(\alpha_{n})$) is  $\{\alpha_{j}: 2\le j\le n\},$ $\alpha_{2},\alpha_{i}\le  w_{i}(\alpha_{n})$=$w_{0,S\setminus\{\alpha_{2},\alpha_{i}\}}(w_{0, S\setminus\{\alpha_{2}\}}v_{n-1}(\alpha_{n})).$ Therefore, $w_{0,S\setminus\{\alpha_{2},\alpha_{i}\}}(w_{0, S\setminus\{\alpha_{2}\}}v_{n-1}(\alpha_{n}))$ is a non-simple positive root for all $3\le i\le n.$ 
	\end{proof}
	
	\begin{lemma}\label{lem5.3}
		We have
		\begin{itemize}
			
			\item[(i)] $w_{i}^{-1}(\alpha_{i})$ is a negative root for all $2\le i\le n.$ 
			
			\item[(ii)] For $2\le i\le n$ and $1\le j\le n$ such that $j\neq i,$ $w_{i}^{-1}(\alpha_{j})$ is a positive root. 
			
		\end{itemize}
	\end{lemma}
	\begin{proof}
		Proof of (i): Note that for every $2\le i \le n,$ $x_{i}^{-1}(\alpha_{i})=(w_{0,S\setminus\{\alpha_{2}\}})^{S\setminus\{\alpha_{2},\alpha_{i}\}}(\alpha_{i})$ is negative. Further,  for any $2\le i\le n,$ we have $w_{i}=x_{i}v_{n-1}$ and $\ell(w_{i})=\ell(x_{i})+\ell(v_{n-1}).$ Hence, $w_{i}^{-1}(\alpha_{i})$ is a negative root for $2\le i\le n.$
		
		Proof of (ii): For $i=2$  we have $w_{i}=v_{n-1}.$ Therefore, by Lemma \ref{lem5.1}(ii),(iii), $w_{i}^{-1}(\alpha_{j})$ is a positive root for $j\neq 2.$ 
		
		For any $3\le i\le n,$ and $j\neq 2,i$, we observe that $x_{i}^{-1}(\alpha_{j})$ is a positive root whose support does not contain $\alpha_{2}.$ Further, by Lemma \ref{lem5.1}(ii),(iii), $R^{+}(v_{n-1}^{-1})\subseteq\{\beta\in R^{+}:  \alpha_{2}\le \beta  \}.$ Therefore,  $w_{i}^{-1}(\alpha_{j})=(x_{i}v_{n-1})^{-1}(\alpha_{j})$ is a positive root for $j\neq 2,i.$
		On the other hand, by using Lemma \ref{lem 2.3} we have $w_{0,S\setminus\{\alpha_{2},\alpha_{i}\}}(\alpha_{2})=\alpha_{1}+\cdots +\alpha_{i-1}.$ Note that   $w_{0,S\setminus\{\alpha_{2}\}}(\alpha_{k})=-\alpha_{k}$ for $ k\neq 2.$ Further, by Lemma \ref{Lemma 2.1}, we have $w_{0,S\setminus\{\alpha_{2}\}}(\alpha_{2})=\alpha_{0}-\alpha_{2}.$ Hence, we have $w_{0,S\setminus\{\alpha_{2}\}}(\alpha_{1}+\cdots+ \alpha_{i-1})=\alpha_{0}-(\alpha_{1}+\cdots +\alpha_{i-1}).$ Thus $x_{i}^{-1}(\alpha_{2})=\alpha_{0}-(\alpha_{1}+\cdots +\alpha_{i-1}).$ Further, by using Lemma \ref{lem5.1}(iii) $v_{n-1}^{-1}(x_{i}^{-1}(\alpha_{2}))=\alpha_{i-2}+\cdots +\alpha_{n-2}.$ Therefore, $w_{i}^{-1}(\alpha_{j})$ is a positive root for all $3\le i\le n$ and $1\le j\le n$ such that $j\neq i.$  
	\end{proof}   
	
	\begin{lemma}\label{cor 5.5} Let $v\in W.$
		$H^1(s_{i}, H^0(v, \mathfrak{b}))=0$ for any $1\le i\le n-1.$
	\end{lemma}
	\begin{proof} Let $\mu \neq -\alpha_{i}$ be such that $H^0(v,\mathfrak{b})_{\mu}\neq 0.$ Then
		$\langle \mu,\alpha_{i} \rangle =1,0,-1$ for $i\neq n,$ as $\alpha_{i}$ is a long root. On the other hand, if $H^0(v,\mathfrak{b})_{-\alpha_{i}}\neq 0$ then by Lemma \ref{lem 5.4}, $\mathbb{C}h(\alpha_{i})\oplus \mathbb{C}_{-\alpha_{i}}$ is a $B_{\alpha_{i}}$ submodule of $H^0(v,\mathfrak{b}).$ Therefore, by Lemma \ref{lem 3.1}(3), we have $H^1(s_{i},H^0(v,\mathfrak{b}))=0$ for $i\neq n.$
	\end{proof}
	\begin{lemma}\label{lem5.6}
		Let $v_{n-1}$ be as above. Then we have $H^1(v_{n-1},\mathfrak{b})=0.$ 
	\end{lemma}
	\begin{proof}	
		Recall that $v_{n-1}=(s_{2}s_{1}s_{3}\cdots s_{n-1}s_{n})(s_{2}s_{3}\cdots s_{n-1}).$ To prove the lemma we first show that $H^1(s_{n}s_{n-1},\mathfrak{b})=0.$ By Lemma \ref{cor 5.5} we have $H^1(s_{n-1},\mathfrak{b})=0.$ Let $\mu $ be such that $H^0(s_{n-1},\mathfrak{b})_{\mu}\neq 0$ and $\langle\mu ,\alpha_{n} \rangle =-2,$ then $\mu =-(\alpha_{i}+\cdots +\alpha_{n-1}+2\alpha_{n})$ for some $1\le i\le n-1$ or $\mu =-\alpha_{n}.$ 
		
		If  $\mu=-\alpha_{n},$ then by Lemma \ref{lem 5.4}, the two dimensional $B_{\alpha_{n}}$ module $\mathbb{C}h(\alpha_{n})\oplus \mathbb{C}_{-\alpha_{n}}=V\otimes \mathbb{C}_{-\omega_{n}}$ is a direct summand of $H^0(s_{n-1},\mathfrak{b}),$ where $V$ is the standard two dimensional $\hat{L}_{\alpha_{n}}$-module.
		
		If $\mu=-(\alpha_{i}+\cdots +\alpha_{n-1}+2\alpha_{n})$ for some $1\le i\le n-2,$ then $\mathbb{C}_{\mu+\alpha_{n}}$ is an one dimensional $B_{\alpha_{n-1}}$-summand of $\mathfrak{b}$ which is in fact an $\hat{L}_{\alpha_{n-1}}$-module, therefore, by Lemma \ref{lem 3.1}(1) $H^0(s_{n-1},\mathfrak{b})_{\mu+\alpha_{n}}\neq 0.$ Hence, $\mathbb{C}_{\mu}\oplus\mathbb{C}_{\mu+\alpha_{n}}=V\otimes\mathbb{C}_{-\omega_{n}}$ is a two dimensional $B_{\alpha_{n}}$-summand of $H^0(s_{n-1},\mathfrak{b})$ where $V$ is the standard two dimensional $\hat{L}_{\alpha_{n}}$-module.
		
		Otherwise, we have $\mu=-(\alpha_{n-1}+2\alpha_{n}).$  Since $\mathbb{C}_{-\alpha_{n}}\oplus \mathbb{C}_{\mu+\alpha_{n}}$ is the standard two dimensional $B_{\alpha_{n-1}}$-summand of $\mathfrak{b}$ which is in fact an $\hat{L}_{\alpha_{n-1}}$-module, by Lemma \ref{lem 3.1}(1) $H^0(s_{n-1},\mathfrak{b})_{\mu+\alpha_{n}}\neq 0.$ Hence, $\mathbb{C}_{\mu}\oplus\mathbb{C}_{\mu+\alpha_{n}}=V\otimes\mathbb{C}_{-\omega_{n}}$ is a two dimensional $B_{\alpha_{n}}$-summand of $H^0(s_{n-1},\mathfrak{b})$ where $V$ is the standard two dimensional $\hat{L}_{\alpha_{n}}$-module. Therefore, combining the above discussion together with Lemma \ref{lem 3.1}(3), we have $H^1(s_{n}s_{n-1},\mathfrak{b})=0.$
		
		Since $H^1(v_{n-1},\mathfrak{b})$ is independent of the choice of the reduced expression of $v_{n-1},$ we use the reduced expression $(s_{2}s_{1}s_{3}\cdots s_{n-1})(s_{2}s_{3}\cdots s_{n-2}s_{n}s_{n-1})$ of $v_{n-1}$ to compute $H^1(v_{n-1},\mathfrak{b}).$
		
		Let $u=(s_{2}s_{1}s_{3}\cdots s_{n-1})(s_{2}s_{3}\cdots s_{n-2}).$ Since $H^1(s_{n}s_{n-1},\mathfrak{b})=0,$ therefore by using SES, Lemma \ref{cor 5.5} repeatedly we have $H^1(v_{n-1},\mathfrak{b})=0,$
	\end{proof}
	
	We note that by \cite[Lemma 6.2, p.779]{Kan}, $H^j(v_{n-1}, \mathfrak{b})=0$  for $j\ge 2.$ Now, we show that $H^j(v_{n-1}, \mathfrak{b})=0$ for $j=0,1.$
	\begin{corollary}\label{cor5.7}
		Let $v_{n-1},w_{i}$ $(2\le i\le n)$ be as above. Then we have   
		\begin{enumerate}
			\item [(i)] $H^j(v_{n-1},\mathfrak{b})=0$ for $j=0,1.$
			
			\item [(ii)] $H^j(v_{n-1},\alpha_{n})=0$ for $j=0,1.$
			
			\item [(iii)] $H^j(v_{n-1},\mathfrak{p}_{n})=0$ for $j=0,1.$
			
			\item [(iv)] $H^j(w_{i},\mathfrak{p}_{n})=0$ for $j=0,1$ and $2\le i\le n.$
		\end{enumerate}
	\end{corollary}
	\begin{proof}
		Proof of (i): By Lemma \ref{lem5.6}, $H^{1}(v_{n-1},\mathfrak{b})=0.$ On the other hand, since $v_{n-1}^{-1}(\alpha_{0})=-\alpha_{n-1},$ by the proof of \cite[Theorem 4.1, p.771]{Kan} we have $H^0(v_{n-1},\mathfrak{b})=0.$ 
		
		Proof of (ii): Since $\langle \alpha_{n},\alpha_{n-1}\rangle=-1,$ by Lemma \ref{lem 3.1}(3) , $H^j(v_{n-1},\alpha_{n})=0$ for $j=0,1.$
		
		Proof of (iii): Consider the exact sequence
		\begin{center}
			$0\longrightarrow \mathfrak{b}\longrightarrow\mathfrak{p}_{n}\longrightarrow\mathbb{C}_{\alpha_{n}}\longrightarrow0$
		\end{center}
		of $B$-modules.
		
		Then we have the following long exact sequence 
		\begin{center}
			$0\longrightarrow H^0(v_{n-1},\mathfrak{b})\longrightarrow H^0(v_{n-1},\mathfrak{p}_{n})\longrightarrow H^0(v_{n-1},\alpha_{n})\longrightarrow$ 	$H^1(v_{n-1}, \mathfrak{b})\longrightarrow H^1(v_{n-1},\mathfrak{p}_{n})\longrightarrow H^1(v_{n-1},\alpha_{n})\longrightarrow H^2(v_{n-1},\mathfrak{b})\longrightarrow\cdots$
		\end{center}
		of $B$-modules.
		
		Therefore, by using $(i)$ and $(ii),$ proof of $(iii)$ follows.

		Proof of (iv): 
		Since $w_{i}=x_{i}v_{n}$ is such that $\ell(w_{i})=\ell(x_{i})+\ell(v_{n}),$ by using (iii) together with SES it follows that $H^j(w_{i}, \mathfrak{p}_{n})=0$ for all $j=0,1$ and $2\le i\le n.$
	\end{proof}
	\begin{proposition}\label{prop5.8}
		We have $w_{i}\in W^{P_{n}}$ and $P_{i}=Aut^0(X_{P_{n}}(w_{i}))$ for all $2\le i\le n.$
	\end{proposition}
	\begin{proof}
		Recall that $w_{i}=w_{0,S\setminus \{\alpha_{i}, \alpha_{2}\}} w_{0, S\setminus \{\alpha_{2}\}} v_{n-1}$ for $2\le i\le n.$  By Lemma \ref{lem5.2}(iii), we conclude  that $w_{i}(\alpha_{n})$ is a non-simple positive root.
		On the other hand, by Lemma \ref{lem5.3}, we have the following:
		
		\begin{enumerate}
			\item [(i)] $w_{i}^{-1}(\alpha_{i})$ is a negative root.
			
			\item [(ii)] $w_{i}^{-1}(\alpha_{j})$ is a positive root for $j\neq i.$
		\end{enumerate}
		
		Thus $P_{i}$ is the stabilizer of $X_{P_{n}}(w_{i})$ in $G.$ Since $v_{n-1}^{-1}(\alpha_{0})$ is a negative root, $w_{i}^{-1}(\alpha_{0})=v_{n-1}^{-1}(\alpha_{0})$ (as$~w_{0,S\setminus\{\alpha_{2}\}}w_{0,S\setminus\{\alpha_{2}, \alpha_{i}\}}(\alpha_{0})=\alpha_{0}$) is a negative root. Therefore, by using \cite[Theorem 6.6, p.781]{Kan} the natural homomorphism $\varphi:P_{i}\longrightarrow Aut^0(X_{P_{n}}(w_{i}))$ is an injective homomorphism of algebraic groups.
		
		Now consider the following SES 
		\begin{center}
			$0\longrightarrow \mathfrak{p}_{n}\longrightarrow \mathfrak{g}\longrightarrow \mathfrak{g/p}_{n}\longrightarrow 0$	
		\end{center}
		of $B$-modules.
		
		Thus, we have the following long exact sequence 
		\begin{center}
			$0\longrightarrow H^0(w_{i},\mathfrak{p}_{n})\longrightarrow H^0(w_{i},\mathfrak{g})\longrightarrow H^0(w_{i}, \mathfrak{g/p}_{n})\longrightarrow$ 
			
			$H^1(w_{i},\mathfrak{p}_{n})\longrightarrow H^1(w_{i},\mathfrak{g})\longrightarrow H^1(w_{i}, \mathfrak{g/p}_{n})\longrightarrow\cdots $	
		\end{center}
		of $B$-modules.  
		
		Since $H^0(w_{i},\mathfrak{g})=\mathfrak{g},$ by  using Corollary \ref{cor5.7}(iv),  we have $H^0(w_{i},\mathfrak{g/p}_{n})=\mathfrak{g}.$ Further, since $d\varphi(\mathfrak{p}_{i})\subseteq$Lie$(Aut^0(X_{P_{n}}(w_{i})))\subseteq H^0(w_{i},\mathfrak{g/p}_{n})=\mathfrak{g}$ and the base field is $\mathbb{C},$ $Aut^0(X_{P_{n}}(w_{i}))$ is a
		closed subgroup of $G$ containing $P_{i}$ (see \cite[Theorem 12.5, p.85 and Theorem 13.1, p.87]{Hum2}). Hence, $Aut^0(X_{P_{n}}(w_{i}))$ is the stabilizer of $X_{P_{n}}(w_{i})$ in $G.$ Thus, we have $Aut^0(X_{P_{n}}(w_{i}))=P_{i}.$ 
	\end{proof}

	\section{ $G$ is of type $C_{n}$ $(n\ge 3)$}
	In this section, we assume that $G$ is of type $C_{n}.$ Further, we prove that for any $1\le i\le n-1,$ there exists a Schubert variety $X_{P_{n-1}}(w_{i})$ in $G/P_{n-1}$ such that $P_{i}=Aut^0(X_{P_{n-1}}(w_{i})).$ 
	
	Since $G$ is of type $C_{n},$  we have $\alpha_{0}=2\alpha_{1}+\cdots +2\alpha_{n-1}+\alpha_{n}=2\omega_{1}.$ Recall that by Corollary \ref{cor 2.2}, there exists a unique element $v_{n}$ in $W$ of minimal length such that $v_{n}^{-1}(\alpha_{0})=-\alpha_{n}$
	
	\begin{lemma}\label{Lemma 4.1}
		Then $v_{n}$ satisfies the following:
		\item[(i)] $v_{n}=s_{1}s_{2}\cdots s_{n}.$
		
		\item[(ii)] $v_{n}^{-1}(\alpha_{1})=-(\alpha_{1}+\alpha_{2}+\cdots +\alpha_{n-1}+\alpha_{n}),$ $v_{n}^{-1}(\alpha_{j})=\alpha_{j-1}$ for all $2\le j\le n-1,$ $v_{n}^{-1}(\alpha_{n})=\alpha_{n}+2\alpha_{n-1}.$
	\end{lemma}
	\begin{proof}
		Proof of (i): Let ${v}_{n}'=s_{1}s_{2}\cdots s_{n}.$ Then by the usual calculation we have ${{v}_{n}'}^{-1}(\alpha_{0})=-\alpha_{n}.$ Moreover, since $\ell({v}_{n}')=\ell(v_{n}),$ by Corollary \ref{cor 2.2}, we have $v_{n}={v}_{n}'.$
		
		Proof of (ii): Follows from the usual calculation. 
	\end{proof}
	Let $x_{i}=w_{0,S\setminus\{\alpha_{1},\alpha_{i}\}}w_{0,S\setminus\{\alpha_{1}\}}=((w_{0,S\setminus\{\alpha_{1}\}})^{S\setminus\{\alpha_{1}, \alpha_{i}\}})^{-1}$ and $w_{i}=x_{i}v_{n}$ for all $1\le i\le n-1.$ Note that $R^{+}(v_{n}^{-1})\subseteq\{\beta\in R^{+}:  \alpha_{1}\le \beta  \}.$ On the other hand, $R^{+}(x_{i})\subseteq \mathbb{Z}_{\ge 0}(S\setminus\{\alpha_{1}\})\cap R^{+}$ for all $1\le i\le n-1.$ Therefore, $R^{+}(v_{n}^{-1})\cap R^{+}(x_{i})=\emptyset.$ Hence, we have $\ell(w_{i})=\ell(x_{i})+\ell(v_{n})$ for all $1\le i\le n-1.$ 
	\begin{lemma}\label{Lemma 4.2}
		Then we have
		\begin{itemize}
			
			\item[(i)] $v_{n}(\alpha_{n-1})=\alpha_{1}+\cdots +\alpha_{n}.$
			
			\item[(ii)] $w_{0, S\setminus \{\alpha_{1}\}}v_{n}(\alpha_{n-1})=\alpha_{1}+\cdots +\alpha_{n-1}.$
			
			\item[(iii)] $w_{i}(\alpha_{n-1})$ is a non-simple positive root for all $1\le i\le n-1.$
		\end{itemize}
	\end{lemma}
	\begin{proof}
		Proof of (i): follows from the usual calculation.
		
		Proof of (ii): By (i), we have $v_{n}(\alpha_{n-1})=\alpha_{1}+\cdots +\alpha_{n}.$ Note that $w_{0, S\setminus\{\alpha_{1}\}}(\alpha_{j})=-\alpha_{j}$ for all $2\le j\le n.$ Further, by Lemma \ref{Lemma 2.1}, we have $w_{0,S\setminus\{\alpha_{1}\}}(\alpha_{1})=\alpha_{0}-\alpha_{1}.$Therefore, we have $w_{0,S\setminus\{\alpha_{1}\}}v_{n}(\alpha_{n-1})=\alpha_{1}+\cdots +\alpha_{n-1}.$
		
		Proof of (iii): By (ii), we have $w_{0,S\setminus\{\alpha_{1}\}}v_{n}(\alpha_{n-1})=\alpha_{1}+\cdots +\alpha_{n-1}.$ Since the support of $w_{0, S\setminus\{\alpha_{1}\}}v_{n}(\alpha_{n-1})$ is equal to $\{\alpha_{j}: 1\le j\le n-1\},$ $w_{i}(\alpha_{n-1})=w_{0,S\setminus\{\alpha_{1},\alpha_{i}\}}(w_{0, S\setminus\{\alpha_{1}\}}v_{n}(\alpha_{n-1}))$ is a non-simple positive root for all $1\le i\le n-1.$  
	\end{proof}
	
	Then we have
	
	\begin{lemma}\label{lemma 6.3}
		
		\begin{itemize}
			
			\item[(i)] $w_{i}^{-1}(\alpha_{i})$ is a negative root for all $1\le i\le n-1.$ 
			
			\item[(ii)] For $1\le i\le n-1$ and $1\le j\le n$ such that $j\neq i,$ 
			$w_{i}^{-1}(\alpha_{j})$ is a positive root. 
		\end{itemize}
	\end{lemma}
	\begin{proof}
		Proof of (i): Note that $x_{i}^{-1}(\alpha_{i})=(w_{0,S\setminus\{\alpha_{1}\}})^{S\setminus\{\alpha_{1},\alpha_{i}\}}(\alpha_{i})$ is negative. Further,  for any $1\le i\le n-1,$ $w_{i}=x_{i}v_{n}$ and $\ell(w_{i})=\ell(x_{i})+\ell(v_{n}).$ Hence, $w_{i}^{-1}(\alpha_{i})$ is a negative root for $1\le i\le n-1.$
		
		Proof of (ii): For any $1\le i\le n-1,$ and $j\neq 1,i$, we observe that $x_{i}^{-1}(\alpha_{j})$ is a positive root whose support does not contain $\alpha_{1}.$ Further, $R^{+}(v_{n}^{-1})\subseteq\{\beta\in R^{+}:  \alpha_{1}\le \beta  \}.$ Therefore, $w_{i}^{-1}(\alpha_{j})$ is a positive root for $j\neq 1,i.$
		On the other hand, for  $i\ge 3$ by using Lemma \ref{lem 2.3}, we have $w_{0,S\setminus\{\alpha_{1},\alpha_{i}\}}(\alpha_{1})=\alpha_{1}+\cdots +\alpha_{i-1}.$ Note that $w_{0,S\setminus\{\alpha_{1}\}}(\alpha_{k})=-\alpha_{k}$ for $2\le k\le n.$ Further, by Lemma \ref{Lemma 2.1}, we have $w_{0,S\setminus\{\alpha_{1}\}}(\alpha_{1})=\alpha_{0}-\alpha_{1}.$ Hence, we have $w_{0,S\setminus\{\alpha_{1}\}}(\alpha_{1}+\cdots \alpha_{i-1})=\alpha_{0}-(\alpha_{1}+\cdots +\alpha_{i-1}).$ Thus $x_{i}^{-1}(\alpha_{1})=\alpha_{0}-(\alpha_{1}+\cdots +\alpha_{i-1}).$ Further, by using Lemma \ref{Lemma 4.1}(ii) $v_{n}^{-1}(x_{i}^{-1}(\alpha_{1}))=\alpha_{i-1}+\cdots +\alpha_{n-1}.$ For $i=2,$ by using Lemma \ref{Lemma 2.1}, we have $x_{i}^{-1}(\alpha_{1})=\alpha_{0}-\alpha_{1}.$ Further, by using Lemma \ref{Lemma 4.1}, we have $v_{n}^{-1}(\alpha_{0}-\alpha_{1})=\alpha_{1}+\cdots +\alpha_{n-1}.$ Therefore, $w_{2}^{-1}(\alpha_{1})$ is a positive root.
	\end{proof}
	Here, we observe that all positive roots of $C_{n}$ $(n\ge 3),$ are of the following form
	\begin{itemize}
		\item  $\alpha_{i}$ for all $1\le i\le n.$
		
		\item  $\alpha_{i}+\cdots +\alpha_{j}$ for all $1\le i<j\le n.$
		
		\item  $2\alpha_{i}+\cdots +2\alpha_{n-1}+\alpha_{n}$ for all $1\le i\le n-1.$
		
		\item  $\alpha_{i}+\cdots +\alpha_{j-1}+2\alpha_{j}+\cdots+2\alpha_{n-1}+\alpha_{n}$ for all $1\le i<j\le n-1.$
	\end{itemize}
	
	Let $\beta_{j}=2(\alpha_{j}+\cdots +\alpha_{n-1})+\alpha_{n}$ for all $1\le j\le n-1.$ Let  $\beta_{i,j}=\alpha_{i}+\cdots+\alpha_{j-1}+2(\alpha_{j}+\cdots +\alpha_{n-1})+\alpha_{n}$ for all $1\le i< j\le n-1.$ Let $\gamma_{i}=\alpha_{i}+\cdots +\alpha_{n}$ for all $1\le i\le n.$ 
	\begin{lemma}\label{Lemma 4.4} 
		$H^0(s_{k}s_{k+1}\cdots s_{n},\mathfrak{b})_{\mu}\neq 0$ for  $\mu\in \{-\beta_{i}, \beta_{i,i+1},\ldots, -\beta_{i,n-1}, -\gamma_{i}: 1\le i \le k-1 \}$ and $H^0(s_{k}s_{k+1}\cdots s_{n},\mathfrak{b})_{\mu}=0$ for $\mu \in\{-\beta_{k},-\beta_{k,k+1},\ldots,-\beta_{k,n-1},-\gamma_{k}\},$ where  $2\le k\le n-2.$
	\end{lemma}
	\begin{proof}
		We prove the lemma by descending induction.
		
		First we prove the base case for $k=n-1.$ Since $\mathbb{C}_{-\beta_{i,n-1}}$'s and $\mathbb{C}_{-\beta_{n-1}}$ are one dimensional indecomposable $\hat{L}_{\alpha_{n}}$-summands of $\mathfrak{b}$, by Lemma \ref{lem 3.1},  $H^0(s_{n},\mathfrak{b})_{-\beta_{i,n-1}}\neq 0$ for all $1\le i \le  n-2,$ and $H^0(s_{n},\mathfrak{b})_{-\beta_{n-1}}\neq 0.$ Since  $\mathbb{C}_{-\gamma_{i}+\alpha_{n}}\oplus \mathbb{C}_{-\gamma_{i}}$'s are the standard two dimensional indecomposable $\hat{L}_{\alpha_{n}}$-summands of $\mathfrak{b},$ by Lemma \ref{lem 3.1} we have  $H^0(s_{n},\mathfrak{b})_{-\gamma_{i}}\neq 0$  and $H^0(s_{n},\mathfrak{b})_{-\gamma_{i}+\alpha_{n}}\neq 0$ for all $1\le i\le n-1.$ Then $\mathbb{C}_{-\gamma_{i}}\oplus \mathbb{C}_{-\beta_{i,n-1}}$'s are standard two dimensional irreducible $\hat{L}_{\alpha_{n-1}}$-summands of $H^0(s_{n},\mathfrak{b})$ for all $1\le i\le n-2.$ Therefore, by Lemma \ref{lem 3.1} $H^0(s_{n-1}s_{n},\mathfrak{b})_{-\beta_{i,n-1}}\neq 0,$ and $H^0(s_{n-1}s_{n},\mathfrak{b})_{-\gamma_{i}}\neq 0$ for all $1\le i\le n-2.$ Further, since $\langle -\beta_{i,j+1},\alpha_{n-1}\rangle =0$ for all $1\le i\le  j\le n-1$ such that $i\le n-3$ by Lemma \ref{lem 3.1}, we have   $H^0(s_{n-1}s_{n},\mathfrak{b})_{-\beta_{i,j+1}}\neq 0$ for all $1\le i\le j\le n-3.$  Moreover, since $\mathbb{C}_{-\beta_{n-1}}\oplus \mathbb{C}_{-\beta_{n-1}+\alpha_{n-1}}=V\otimes \mathbb{C}_{-\omega_{n-1}},$ where $V$ is the standard two dimensional irreducible module, by Lemma \ref{lem 3.1} we have $H^0(s_{n-1}s_{n},\mathfrak{b})_{-\beta_{n-1}}=0$ and $H^0(s_{n-1}s_{n},\mathfrak{b})_{-\beta_{n-1}+\alpha_{n-1}}=0.$
		Therefore, combining the above discussion we have the following $H^0(s_{n-1}s_{n},\mathfrak{b})_{\mu}\neq 0$ for all $\mu \in M,$ where $M=\{-\beta_{i},-\beta_{i,i+1},\ldots,-\beta_{i,n-1},-\gamma_{i}: 1\le i\le n-2\}$ and $H^0(s_{n-1}s_{n},\mathfrak{b})_{-\beta_{n-1}}=0,H^0(s_{n-1}s_{n},\mathfrak{b})_{-\gamma_{n-1}}=0.$ Now by induction hypothesis we assume that $H^0(s_{k+1}\cdots s_{n-1}s_{n},\mathfrak{b})_{\mu}\neq 0$ for all $\mu\in  \{-\beta_{i},-\beta_{i,i+1},\ldots, -\beta_{i,n-1},-\gamma_{i}:  1\le i\le k\}$ and 
		$H^0(s_{k+1}\cdots s_{n},\mathfrak{b})_{\mu}=0$ for all $\mu \in \{-\beta_{k+1},-\beta_{k+1,k+2},\ldots,-\beta_{k+1,n-1},-\gamma_{k+1} \},$ where $k+1\le n-3.$

		Then we note that $\langle -\beta_{i}, \alpha_{k} \rangle =0,$ $\langle-\gamma_{i},\alpha_{k}  \rangle =0$ for all $i\le k-1.$ Also, $\langle -\beta_{i,k}, \alpha_{k} \rangle =-1,$ $\langle -\beta_{i,k+1}, \alpha_{k} \rangle =1$ for all $1\le i\le k-1,$ and  $\langle -\beta_{i,j+1}, \alpha_{k} \rangle =0$ for all $1\le i\le j\le n-2$ such that $i\le k-1$ and $j\neq k-1,k.$ By using Lemma \ref{lem 3.2}, we have  $ \mathbb{C}_{-\beta_{k}}\oplus \mathbb{C}_{-\beta_{k}+\alpha_{k}}=V\otimes \mathbb{C}_{-\alpha_{k}}$ where $V$ is the standard two dimensional indecomposable $B_{\alpha_{k}}$-submodule of $H^0(s_{k+1}\cdots s_{n-1}s_{n},\mathfrak{b}).$ On the other hand, since $H^0(s_{k+1}\cdots s_{n-1}s_{n},\alpha_{n})_{\mu_{j}}=0$ for all $\mu_{j}=-\beta_{k+1,j+1},-\gamma_{k+1}$ where $k+2 \le j\le n-2,$ and $\langle -\beta_{k,j+1}, \alpha_{k} \rangle=-1,$ $\langle -\gamma_{k}, \alpha_{k} \rangle=-1$ for all $k+1\le j\le n-2,$  $\mathbb{C}_{-\beta_{k,j+1}}$'s and $\mathbb{C}_{-\gamma_{k}}$ are  one dimensional indecomposable $\hat{L}_{\alpha_{k}}$-summands of $H^0(s_{k+1}\cdots s_{n-1}s_{n},\mathfrak{b})$ for all $k+2 \le j\le n-2.$ Therefore, by combining the above discussion together with Lemma \ref{lem 3.1} we have  $H^0(s_{k}s_{k+1}\cdots s_{n},\mathfrak{b})_{\mu}\neq 0$ for all  $\mu\in \{-\beta_{i}, -\beta_{i,i+1},\ldots, -\beta_{i,n-1},-\gamma_{i}: 1\le i \le k \}$ and
		$H^0(s_{k}s_{k+1}\cdots s_{n},\mathfrak{b})_{\mu}=0$ for all $\mu \in \{-\beta_{k+1},-\beta_{k+1,k+2},\ldots,-\beta_{k+1,n-1},-\gamma_{k+1} \},$ where  $2\le k\le n-2.$
	\end{proof}
	
	\begin{lemma}\label{cor 4.5}Let $u_{k}=s_{k}\cdots s_{n}$ for all $1\le k\le n-1.$ Then we have $H^1(u_{k},\mathfrak{b})=0.$ In particular, we have $H^1(v_{n},\mathfrak{b})=0$ as $u_{1}=v_{n}.$
	\end{lemma}
	\begin{proof}
		We prove by the descending induction on $k.$ We first prove the statement for  $k=n-1,$ i.e., $H^1(s_{n-1}s_{n},\mathfrak{b})=0.$ Since $\alpha_{n}$ is a long root, $\langle \beta, \alpha_{n}\rangle = 1,0$ or $-1$ for any root $\beta.$ Therefore, by using Lemma \ref{lem 3.1},  we have $H^1(s_{n},\mathfrak{b})=0.$ Hence, we have $H^0(s_{n-1},H^1(s_{n},\mathfrak{b}))=0.$ On the other hand,  $\mathbb{C}h(\alpha_{n-1})\oplus\mathbb{C}_{-\alpha_{n-1}}$ and $\mathbb{C}_{-(\alpha_{n-1}+\alpha_{n})}\oplus \mathbb{C}_{-(2\alpha_{n-1}+\alpha_{n})}$ are indecomposable $B_{\alpha_{n-1}}$-submodules of  $H^0(s_{n},\mathfrak{b}).$  Then by Lemma \ref{lem 3.2}, $\mathbb{C}h(\alpha_{n})\oplus\mathbb{C}_{-\alpha_{n}}=V\otimes\mathbb{C}_{-\omega_{n-1}} $ and $\mathbb{C}_{-(\alpha_{n-1}+\alpha_{n})}\oplus \mathbb{C}_{-(2\alpha_{n-1}+\alpha_{n})}=V\otimes \mathbb{C}_{-\omega_{n-1}},$ where $V$ is the standard two dimensional $\hat{L}_{\alpha_{n-1}}$-modules. Further, we have $\langle \beta , \alpha_{n-1} \rangle =0,1,$or $-1$ for any $\beta$ such that $H^0(s_{n},\mathfrak{b})_{\beta}\neq 0$ and $\beta \neq -\alpha_{n-1},-(2\alpha_{n-1}+\alpha_{n}).$ Thus, by using Lemma \ref{lem 3.1}, we have $H^1(s_{n-1},H^0(s_{n},\mathfrak{b}))=0.$ Therefore, we have $H^1(s_{n-1}s_{n},\mathfrak{b})=0.$ 
		
		Now by induction hypothesis e assume that $H^1(u_{j},\mathfrak{b})=0$ for all $k+1\le j\le n-1.$

		Note that $-\beta_{k}$ is the only negative root other than $-\alpha_{k}$ such that $\langle -\beta_{k},\alpha_{k} \rangle =-2.$ On the other hand, by Lemma \ref{Lemma 4.4}, $\mathbb{C}h(\alpha_{k})\oplus \mathbb{C}_{-\alpha_{k}}=V\otimes\mathbb{C}_{-\omega_{k}}$ and  $\mathbb{C}_{-(\alpha_{k}+\beta_{k})}\oplus \mathbb{C}_{-\beta_{k}}=V\otimes\mathbb{C}_{-\omega_{k}}$ are indecomposable $\hat{B}_{\alpha_{k}}$-summands of $H^0(s_{k+1}\cdots s_{n-1}s_{n},\mathfrak{b})$ where $V$ is the standard two dimensional irreducible $\hat{L}_{\alpha_{k}}$-module. Therefore, $H^1(s_{k},H^0(u_{k+1},\mathfrak{b}))=0.$
		
		By induction hypothesis we have $H^1(u_{k+1},\mathfrak{b})=0.$ Therefore, by combining the above discussion together with using SES
		\begin{center}
			$0\longrightarrow H^1(s_{k},H^0(u_{k+1},\mathfrak{b}))\longrightarrow H^1(u_{k},\mathfrak{b})\longrightarrow H^0(s_{k},H^1(u_{k+1},\mathfrak{b}))\longrightarrow0$ 
		\end{center}
		we have $H^1(u_{k},\mathfrak{b})=0.$ 
	\end{proof}
	
	We note that by \cite[Lemma 6.2, p.779]{Kan}, $H^j(v_{n}, \mathfrak{b})=0$  for $j\ge 2.$ Now we show that $H^j(v_{n}, \mathfrak{b})=0$ for $j=0,1.$
	\begin{corollary}\label{cor 6.6} We have the following
		\begin{itemize}
			\item [(i)] $H^j(v_{n},\mathfrak{b})=0$ for all $j=0,1.$
			
			\item [(ii)] $H^j(v_{n},\alpha_{n-1})=0$ for all $j=0,1.$
			
			\item [(iii)] $H^j(v_{n},\mathfrak{p}_{n-1})=0$ for all $j=0,1.$
			
			\item [(iv)] $H^j(w_{i},\mathfrak{p}_{n-1})=0$ for all $1\le i\le n-1$ and $j=0,1.$
		\end{itemize}
	\end{corollary}
	\begin{proof}
		Proof of (i): By Lemma \ref{cor 4.5}, $H^{1}(v_{n},\mathfrak{b})=0.$ On the other hand, since $v_{n}^{-1}(\alpha_{0})=-\alpha_{n},$ by the proof of \cite[Theorem 4.1, p.771]{Kan}, we have $H^0(v_{n},\mathfrak{b})=0.$ 
		
		Proof of (ii): Since $\langle \alpha_{n-1},\alpha_{n}\rangle=-1,$ by Lemma \ref{lem 3.1}(3), we have $H^j(v_{n},\alpha_{n-1})=0$ for $j=0,1.$

		Proof of (iii):
		Consider the exact sequence
		\begin{center}
			$0\longrightarrow \mathfrak{b}\longrightarrow\mathfrak{p}_{n-1}\longrightarrow\mathbb{C}_{\alpha_{n-1}}\longrightarrow0$
		\end{center}
		of $B$-modules.
		
		Then we have the following long exact sequence 
		\begin{center}
			$0\longrightarrow H^0(v_{n-1},\mathfrak{b})\longrightarrow H^0(v_{n-1},\mathfrak{p}_{n-1})\longrightarrow H^0(v_{n-1},\alpha_{n-1})\longrightarrow$ 	$H^1(v_{n-1}, \mathfrak{b})\longrightarrow H^1(v_{n-1},\mathfrak{p}_{n-1})\longrightarrow H^1(v_{n-1},\alpha_{n-1})\longrightarrow H^2(v_{n-1},\mathfrak{b})\longrightarrow\cdots$
		\end{center}
		of $B$-modules.
		
		Therefore, by using $(i)$ and $(ii),$ proof of $(iii)$ follows.

		Proof of (iv): Since $w_{i}=x_{i}v_{n}$ is such that $\ell(w_{i})=\ell(x_{i})+\ell(v_{n})$ for all $1\le i\le n-1,$ by (iii) and using SES proof of (iv) follows. 
	\end{proof}

	\begin{proposition}\label{prop6.7} We have $w_{i}\in W^{P_{n-1}}$ and $P_{i}=Aut^0(X_{P_{n-1}}(w_{i}))$ for all $1\le i\le n-1.$
	\end{proposition}
	\begin{proof}
		Recall that $w_{i}=w_{0,S\setminus \{\alpha_{1}, \alpha_{i}\}} w_{0, S\setminus \{\alpha_{1}\}} v_{n}$ for $1\le i\le n-1.$  By Lemma \ref{Lemma 4.2}(iii), we conclude  that $w_{i}(\alpha_{n-1})$ is a non-simple positive root.
		
		On the other hand, by Lemma \ref{lemma 6.3}, we have the following:
		
		\begin{itemize}
			\item[(i)]  $w_{i}^{-1}(\alpha_{i})$ is a negative root.\\
			
			\item[(ii)]  $w_{i}^{-1}(\alpha_{j})$ is a positive root for $j\neq i.$
		\end{itemize}

		Thus $P_{i}$ is the stabilizer of $X_{P_{n-1}}(w_{i})$ in $G.$ Since $\alpha_{0}=2\omega_{1},$ and  $v_{n}^{-1}(\alpha_{0})$ is a negative root, $w_{i}^{-1}(\alpha_{0})=v_{n}^{-1}(\alpha_{0})$ (as$~w_{0,S\setminus\{\alpha_{1}\}}w_{0,S\setminus\{\alpha_{1}, \alpha_{i}\}}(\alpha_{0})=\alpha_{0}$) is a negative root. Therefore, by using \cite[Theorem 6.6, p.781]{Kan} the natural homomorphism $\varphi :P_{i}\longrightarrow Aut^0(X_{P_{n-1}}(w_{i}))$ is an injective homomorphism of algebraic groups. 
		
		Now consider the following SES 
		\begin{center}
			$0\longrightarrow \mathfrak{p}_{n-1}\longrightarrow \mathfrak{g}\longrightarrow \mathfrak{g/p}_{n-1}\longrightarrow 0$	
		\end{center}
		of $B$-modules.
		
		Thus, we have the following long exact sequence 
		\begin{center}
			$0\longrightarrow H^0(w_{i},\mathfrak{p}_{n-1})\longrightarrow H^0(w_{i},\mathfrak{g})\longrightarrow H^0(w_{i}, \mathfrak{g/p}_{n-1})\longrightarrow$ 
			
			$H^1(w_{i},\mathfrak{p}_{n-1})\longrightarrow H^1(w_{i},\mathfrak{g})\longrightarrow H^1(w_{i}, \mathfrak{g/p}_{n-1})\longrightarrow\cdots $	
		\end{center}
		of $B$-modules.  
		
		Since $H^0(w_{i},\mathfrak{g})=\mathfrak{g},$ by  using Corollary \ref{cor 6.6}(iv),  we have $H^0(w_{i},\mathfrak{g/p}_{n-1})=\mathfrak{g}.$ Further, since $d\varphi(\mathfrak{p}_{i})\subseteq$Lie$(Aut^0(X_{P_{n-1}}(w_{i})))\subseteq H^0(w_{i},\mathfrak{g/p}_{n-1})=\mathfrak{g}$ and the base field  is $\mathbb{C},$ $Aut^0(X_{P_{n-1}}(w_{i}))$ is a closed subgroup of $G$ containing $P_{i}$ (see \cite[Theorem 12.5, p.85 and Theorem 13.1, p.87]{Hum2}) Hence, $Aut^0(X_{P_{n-1}}(w_{i}))$ is the stabilizer of $X_{P_{n-1}}(w_{i})$ in $G.$ Thus, we have  $Aut^0(X_{P_{n-1}}(w_{i}))=P_{i}.$

	\end{proof}

	\section{$G$ is of type $F_{4}$}
	In this section, we assume that $G$ is of type $F_{4}.$ Further, we prove that for any $1\le i\le 4,$ there exists  a Schubert variety $X_{P_{3}}(w_{i})$ in $G/P_{3}$ such that $P_{i}=Aut^0(X_{P_{3}}(w_{i})).$ 
	
	Since $G$ is of type $F_{4},$ we have $\alpha_{0}=2\alpha_{1}+3\alpha_{2}+4\alpha_{3}+2\alpha_{4}=\omega_{1}.$ Recall that by Corollary \ref{cor 2.2}, there exists a unique element $v_{2}$ of minimal length such that $v_{2}^{-1}(\alpha_{0})=-\alpha_{2}.$
	
	\begin{lemma}\label{lemma 3.1}
		Then $v_{2}$ satisfies the following:
		\begin{itemize}
			\item [(i)] $v_{2}=s_{1}s_{2}s_{3}s_{2}s_{4}s_{3}s_{1}s_{2}.$
			
			\item [(ii)] $v_{2}^{-1}(\alpha_{1})=-(\alpha_{1}+3\alpha_{2}+4\alpha_{3}+2\alpha_{4}),$ $v_{2}^{-1}(\alpha_{2})=\alpha_{2}+2\alpha_{3},$ $v_{2}^{-1}(\alpha_{3})=\alpha_{4},$ and $v_{2}^{-1}(\alpha_{4})=\alpha_{1}+\alpha_{2}+\alpha_{3}.$
		\end{itemize}
	\end{lemma}
	\begin{proof}
		Proof of (i):
		Let $v'_{2}=s_{1}s_{2}s_{3}s_{2}s_{4}s_{3}s_{1}s_{2}.$ Then by the usual calculation we have ${v'_{2}}^{-1}(\alpha_{0})=-\alpha_{2}.$ Since $\ell(v_{2}') =\ell(v_{2}),$ by Corollary \ref{cor 2.2}, we have $v_{2}=v'_{2}.$

		Proof of (ii): By $(i),$ we have $v_{2}^{-1}=s_{2}s_{1}s_{3}s_{4}s_{2}s_{3}s_{2}s_{1}.$ Then by the usual calculation we have $v_{2}^{-1}(\alpha_{1})=-(\alpha_{1}+3\alpha_{2}+4\alpha_{3}+2\alpha_{4}),$ $v_{2}^{-1}(\alpha_{2})=-(\alpha_{2}+2\alpha_{3}),$ $v_{2}^{-1}(\alpha_{3})=\alpha_{4},$ and $v_{2}^{-1}(\alpha_{4})=\alpha_{1}+\alpha_{2}+\alpha_{3}.$
	\end{proof}
	
	Let $x_{i}=w_{0,S\setminus\{\alpha_{1},\alpha_{i}\}}w_{0,S\setminus\{\alpha_{1}\}}=((w_{0,S\setminus\{\alpha_{1}\}})^{S\setminus\{\alpha_{1}, \alpha_{i}\}})^{-1}$ and $w_{i}=x_{i}v_{2}$ for all $1\le i\le 4.$ Note that by Lemma \ref{lemma 3.1}, we have $R^{+}(v_{2}^{-1})\subseteq\{\beta\in R^{+}:  \alpha_{1}\le \beta  \}.$ On the other hand, $R^{+}(x_{i})\subseteq \mathbb{Z}_{\ge 0}(S\setminus\{\alpha_{1}\})\cap R^{+}$ for all $1\le i\le 4.$ Therefore, $R^{+}(v_{2}^{-1})\cap R^{+}(x_{i})=\emptyset.$ Hence, $\ell(w_{i})=\ell(x_{i})+\ell(v_{2})$ for all $1\le i\le 4.$ 
	\begin{lemma}\label{lemma 7.2}
		\item[(i)] $v_{2}(\alpha_{3})=\alpha_{1}+2\alpha_{2}+2\alpha_{3}+\alpha_{4}.$	
		
		\item[(ii)] $w_{0, S\setminus \{\alpha_{1}\}}v_{2}(\alpha_{3})=\alpha_{1}+\alpha_{2}+2\alpha_{3}+\alpha_{4}.$
		
		\item[(iii)] $w_{i}(\alpha_{3})$ is a non-simple positive root for $1\le i\le 4.$
	\end{lemma}
	\begin{proof}
		Proof of (i): By the usual calculation we have $v_{2}(\alpha_{3})=\alpha_{1}+2\alpha_{2}+2\alpha_{3}+\alpha_{4}.$
		
		Proof of (ii): Since $w_{0, S\setminus\{\alpha_{1}\}}(\alpha_{j})=-\alpha_{j}$ for $j\neq 1,$ $w_{0, S\setminus \{\alpha_{1}\}}(\alpha_{1})=\alpha_{0}-\alpha_{1},$ we have $w_{0,S\setminus\{\alpha_{1}\}}v_{2}(\alpha_{3})=\alpha_{0}-(\alpha_{1}+2\alpha_{2}+2\alpha_{3}+\alpha_{4})=\alpha_{1}+\alpha_{2}+2\alpha_{3}+\alpha_{4}.$

		Proof of (iii): For $i=1,$ $w_{1}=v_{2}.$ Therefore, it follows from $(i).$  By $(ii)$ we have $w_{0,S\setminus\{\alpha_{1}\}}v_{2}(\alpha_{3})=\alpha_{1}+\alpha_{2}+2\alpha_{3}+\alpha_{4}.$ Since support of $w_{0,S\setminus\{\alpha_{1}\}}v_{2}(\alpha_{3})$ is $S,$ $w_{i}(\alpha_{3})$ is a non-simple positive root.
	\end{proof}
	\begin{lemma}\label{lemma 7.3}
		$w_{i}^{-1}(\alpha_{i})$ is a negative root and
		$w_{i}^{-1}(\alpha_{j})$ is a positive root for all $1\le i, j\le 4$ such that $i\neq j.$  
	\end{lemma}
	\begin{proof}
		First we note that for $i=1,$ we have $w_{1}=v_{2}.$ Therefore, for $i=1,$ statement follows from Lemma \ref{lemma 3.1}(ii).
		
		On the other hand, for $i\neq 1,$  $x_{i}^{-1}(\alpha_{i})=((w_{0, S\setminus \{\alpha_{1}\}})^{S\setminus \{\alpha_{1},\alpha_{i}\}})(\alpha_{i})$ is a negative root such that support does not contain $\alpha_{1}.$ Therefore, by Lemma \ref{lemma 3.1}(ii), $w_{i}^{-1}(\alpha_{i})=v_{2}^{-1}(x_{i}^{-1}(\alpha_{i}))$ is a negative root for $i\neq 1.$
		
		Further, we have $x_{i}^{-1}(\alpha_{j})$ is a positive root such that its support does not contain $\alpha_{1},$ for all $i\neq j$ and $j\neq 1.$  Therefore, by Lemma \ref{lemma 3.1}(ii), $w_{i}^{-1}(\alpha_{j})=v_{2}^{-1}(x_{i}^{-1}(\alpha_{j}))$ is a positive root for $i\neq j$ and $j\neq 1.$ 
		
		{\bf Claim:} If $i\neq j,$ but $j=1,$ then  $w_{i}^{-1}(\alpha_{j})$ is a positive root. 
		
		{\bf Case I:} Assume that $i=2.$

		\begin{picture}(10,0)
			\thicklines
			\put(2.2,0){\circle*{0.2}}
			\put(2,-0.5){$\alpha_{3}$}
			\put(2.2,0){\line(1,0){1}}
			\put(3.2,0){\circle*{0.2}}
			\put(3.1,-0.5){$\alpha_{4}$}
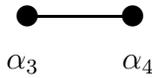
			\put(0,-1.5){Figure 5: Dynkin subdiagram of $F_{4}$ removing the nodes $\alpha_{1},\alpha_{2}.$}
			\put(0,-1.9){}
		\end{picture}
		\vspace{2cm}
		
		Then
		$w_{0,S\setminus\{\alpha_{1},\alpha_{2}\}}(\alpha_{1})=s_{3}s_{4}s_{3}(\alpha_{1})=\alpha_{1}.$ Further, by Lemma \ref{Lemma 2.1}, $w_{0, S\setminus\{\alpha_{1}\}}(\alpha_{1})=\alpha_{0}-\alpha_{1}.$ Moreover, by using Lemma \ref{lemma 3.1}, we have $v^{-1}_{2}(\alpha_{0}-\alpha_{1})=-\alpha_{2}+(\alpha_{1}+3\alpha_{2}+4\alpha_{3}+2\alpha_{4})=\alpha_{1}+2\alpha_{2}+4\alpha_{3}+2\alpha_{4}.$ Therefore, $w_{2}^{-1}(\alpha_{1})$ is a positive root.
		
		{\bf Case II:} Assume that $i=3.$
		
		\begin{picture}(10,0)
			\thicklines
			\put(2,-0.5){$\alpha_{2}$}
			\put(2.2,0){\circle*{0.2}}
			\put(3.2,0){\circle*{0.2}}
			\put(3.1,-0.5){$\alpha_{4}$}
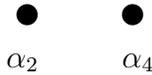
			\put(0,-1.5){Figure 6: Dynkin subdiagram of $F_{4}$ removing the nodes $\alpha_{1},\alpha_{3}.$}
			\put(0,-1.9){}
		\end{picture}
		\vspace{2cm}
		
		Then
		$w_{0,S\setminus\{\alpha_{1},\alpha_{3}\}}(\alpha_{1})=s_{2}s_{4}(\alpha_{1})=\alpha_{1}+\alpha_{2}.$ Further, since $w_{0,S\setminus\{\alpha_{1}\}}(\alpha_{2})=-\alpha_{2},$ and by Lemma \ref {Lemma 2.1}, $w_{0, S\setminus\{\alpha_{1}\}}(\alpha_{1})=\alpha_{0}-\alpha_{1},$ we have $w_{0,S\setminus\{\alpha_{1}\}}(\alpha_{1}+\alpha_{2})=\alpha_{0}-\alpha_{1}-\alpha_{2}.$ Moreover, by using Lemma \ref{lemma 3.1},  $v^{-1}_{2}(\alpha_{0}-\alpha_{1}-\alpha_{2})=-\alpha_{2}+(\alpha_{1}+3\alpha_{2}+4\alpha_{3}+2\alpha_{4})-(\alpha_{2}+2\alpha_{3})=\alpha_{1}+\alpha_{2}+2\alpha_{3}+2\alpha_{4}.$ Therefore, $w_{3}^{-1}(\alpha_{1})$ is a positive root.
		
		{\bf Case: III} Assume that $i=4.$
		
		\begin{picture}(10,0)
			\thicklines
			\put(1.2,0){\circle*{0.2}}
			\put(1.2,-0.1){\line(1,0){1}}
			\put(1,-0.5){$\alpha_{2}$}
			\put(1.2,0.1){\line(1,0){1}}
			\put(1.5,0.2){\line(2,-1){.4}}
			\put(1.5,-0.2){\line(2,1){.4}}
			\put(2.2,0){\circle*{0.2}}
			\put(2,-0.5){$\alpha_{3}$}
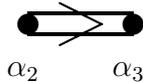
			\put(0,-1.5){Figure 7: Dynkin subdiagram of $F_{4}$ removing the nodes $\alpha_{1},$ $\alpha_{4}.$}
			\put(0,-1.9){}
		\end{picture}
		\vspace{2cm}

		Then
		$w_{0,S\setminus\{\alpha_{1},\alpha_{4}\}}(\alpha_{1})=s_{2}s_{3}s_{2}s_{3}(\alpha_{1})=\alpha_{1}+2\alpha_{2}+2\alpha_{3}.$ Further, since $w_{0,S\setminus\{\alpha_{1}\}}(\alpha_{i})=-\alpha_{i}$ for $i=2,3,$ and by Lemma \ref{Lemma 2.1}, $w_{0, S\setminus\{\alpha_{1}\}}(\alpha_{1})=\alpha_{0}-\alpha_{1},$ we have $w_{0,S\setminus\{\alpha_{1}\}}(\alpha_{1}+2\alpha_{2}+2\alpha_{3})=\alpha_{0}-\alpha_{1}-2\alpha_{2}-2\alpha_{3}.$ Moreover, by using Lemma \ref{lemma 3.1}, $v^{-1}_{2}(\alpha_{0}-\alpha_{1}-2\alpha_{2}-2\alpha_{3})=-\alpha_{2}+(\alpha_{1}+3\alpha_{2}+4\alpha_{3}+2\alpha_{4})-2(\alpha_{2}+2\alpha_{3})-2\alpha_{4}=\alpha_{1}.$ Therefore, $w_{3}^{-1}(\alpha_{1})$ is a positive root.
		
		Thus, combining the above discussion proof of the lemma follows.
	\end{proof}
	
	We note that by \cite[Lemma 6.2, p.779]{Kan}, $H^j(v_{2}, \mathfrak{b})=0$  for $j\ge 2.$ Now we show that $H^j(v_{2}, \mathfrak{b})=0$ for $j=0,1.$ To proceed further, we need to understand $\mathfrak{b}.$ Recall that $F_{4}$ has 24 positive roots. In the following table, we write down explicitly all the positive roots of $F_{4},$ which helps to prove $H^j(v_{2}, \mathfrak{b})=0$ for $j=0,1.$   
	\begin{center}
		\begin{tabular}{ |p{1.3cm}|p{10.3cm}|p{3.5cm}| }
			\hline
			\multicolumn{2}{|c|}{ Positive roots of $F_{4}$} \\
			\hline
			
			Height&Positive roots  \\
			\hline
			$1$ &  $\alpha_{1},~\alpha_{2},~\alpha_{3},~\alpha_{4}~$  \\
			
			\hline	
			2 &
			$\alpha_{1}+\alpha_{2},~\alpha_{2}+\alpha_{3},~\alpha_{3}+\alpha_{4}$ \\
			\hline
			3&  $\alpha_{1}+\alpha_{2}+\alpha_{3},~\alpha_{2}+2\alpha_{3},~\alpha_{2}+\alpha_{3}+\alpha_{4}$     \\
			\hline	
			4&  $\alpha_{1}+\alpha_{2}+\alpha_{3}+\alpha_{4},~\alpha_{1}+\alpha_{2}+2\alpha_{3},~\alpha_{2}+2\alpha_{3}+\alpha_{4}$ \\
			\hline	
			5&	$\alpha_{1}+\alpha_{2}+2\alpha_{3}+\alpha_{4},~\alpha_{1}+2\alpha_{2}+2\alpha_{3},~\alpha_{2}+2\alpha_{3}+2\alpha_{4}$\\
			
			\hline
			
			6 & $\alpha_{1}+2\alpha_{2}+2\alpha_{3}+\alpha_{4},~\alpha_{1}+\alpha_{2}+2\alpha_{3}+2\alpha_{4}$\\
			\hline

			7& $\alpha_{1}+2\alpha_{2}+2\alpha_{3}+2\alpha_{4},~\alpha_{1}+2\alpha_{2}+3\alpha_{3}+\alpha_{4}$\\
			
			\hline
			
			8& $\alpha_{1}+2\alpha_{2}+3\alpha_{3}+2\alpha_{4}$\\
			
			\hline
			
			9& $\alpha_{1}+2\alpha_{2}+4\alpha_{3}+2\alpha_{4}$\\
			
			\hline
			
			10& $\alpha_{1}+3\alpha_{2}+4\alpha_{3}+2\alpha_{4}$\\
			
			\hline
			
			11& $2\alpha_{1}+3\alpha_{2}+4\alpha_{3}+2\alpha_{4}$\\
			
			\hline
		\end{tabular}
		
		\begin{center}
			\bf{Table: Positive roots of $F_{4}.$}
		\end{center}
	\end{center}

	Recall that $v_{2}=s_{1}s_{2}s_{3}s_{2}s_{4}s_{3}s_{1}s_{2}.$ Now we prove the following lemma
	\begin{lemma}\label{Lemma 7.4} Then we have
		$H^{1}(v_{2}, \mathfrak{b})=0.$ 	
	\end{lemma}
	\begin{proof}
		Note that $\mathfrak{b}=\bigoplus\limits_{i=1}^{4}\mathbb{C}h(\alpha_{i})\bigoplus\limits_{\beta \in R^{+}}^{} \mathfrak{g}_{-\beta}.$
		
		Since $\mathbb{C}h(\alpha_{2})\oplus \mathfrak{g}_{-\alpha_{2}}=V\otimes \mathbb{C}_{-\omega_{2}},$ where $V$ is the standard two dimensional irreducible $\hat{L}_{\alpha_{2}}$-module, by Lemma \ref{lem 3.1} we have $H^i(\hat{L}_{\alpha_{2}}/\hat{B}_{\alpha_{2}}, \mathbb{C}h(\alpha_{2})\oplus \mathfrak{g}_{-\alpha_{2}} )=0$ for $i=0,1.$ Further, note that $s_{2}(R^{+}\setminus\{\alpha_{2}\})= R^{+}\setminus\{\alpha_{2}\}.$ Therefore, by using SES we have $H^0(s_{2},\mathfrak{b})=\mathbb{C}h(\alpha_{1})\oplus \mathbb{C}h(\alpha_{3})\oplus \mathbb{C}h(\alpha_{4})\oplus(\mathfrak{g}_{-\alpha_{1}}\oplus \mathfrak{g}_{-\alpha_{3}}\oplus \mathfrak{g}_{-\alpha_{4}})\oplus \bigoplus\limits_{ht(\beta)\ge 2 }^{} \mathfrak{g}_{-\beta}$
		and $H^1(s_{2},\mathfrak{b})=0.$
		
		Since $\mathbb{C}h(\alpha_{1})\oplus \mathfrak{g}_{-\alpha_{1}}=V\otimes \mathbb{C}_{-\omega_{1}},$ where $V$ is the standard two dimensional $\hat{L_{\alpha_{1}}},$ by Lemma \ref{lem 3.1} we have $H^i(\hat{L}_{\alpha_{1}}/\hat{B}_{\alpha_{1}}, \mathbb{C}h(\alpha_{1})\oplus \mathfrak{g}_{-\alpha_{1}} )=0$ for $i=0,1.$ Further, the indecomposable $B_{\alpha_{1}}$-summand $V$ of $H^0(s_{2}, \mathfrak{b})$ containing $\mathfrak{g}_{-(\alpha_{1}+\alpha_{2})}$ is one dimensional and we have $\langle -(\alpha_{1}+\alpha_{2}), \alpha_{1}\rangle=-1.$ Hence, by Lemma \ref{lem 3.1} we have $H^i(s_{1}, V)=0$ for $i=0,1.$  Further, note that $s_{1}(R^{+}\setminus\{\alpha_{1},\alpha_{2}, \alpha_{1}+\alpha_{2}\})= R^{+}\setminus\{\alpha_{1},\alpha_{2}, \alpha_{1}+\alpha_{2}\}.$ Thus, by using SES and Lemma \ref{lem 3.1}, we have $H^0(s_{1}s_{2},\mathfrak{b})= \mathbb{C}h(\alpha_{3})\oplus \mathbb{C}h(\alpha_{4})\oplus (\mathfrak{g}_{-\alpha_{3}}\oplus \mathfrak{g}_{-\alpha_{4}})\oplus  (\mathfrak{g}_{-(\alpha_{2}+\alpha_{3})}\oplus \mathfrak{g}_{-(\alpha_{3}+\alpha_{4})})\oplus \bigoplus\limits_{ht(\beta) \ge 3}^{} \mathfrak{g}_{-\beta}$ and $H^1(s_{1}s_{2}, \mathfrak{b})=0.$

		Note that $\mathbb{C}h(\alpha_{3})\oplus \mathfrak{g}_{-\alpha_{3}}=V\otimes \mathbb{C}_{-\omega_{3}},$ $\mathfrak{g}_{-(\alpha_{2}+\alpha_{3})}\oplus \mathfrak{g}_{-(\alpha_{2}+2\alpha_{3})}=V\otimes \mathbb{C}_{-\omega_{3}},$ $\mathfrak{g}_{-(\alpha_{1}+\alpha_{2}+\alpha_{3})}\oplus \mathfrak{g}_{-(\alpha_{1}+\alpha_{2}+2\alpha_{3})}=V\otimes \mathbb{C}_{-\omega_{3}},$  where $V$ is the standard two dimensional $\hat{L_{\alpha_{3}}}$-module. Further, $s_{3}$ permutes $R^{+}\setminus\{\alpha_{1},\alpha_{2},\alpha_{3}, \alpha_{1}+\alpha_{2},\alpha_{2}+\alpha_{3}, \alpha_{2}+2\alpha_{3},\alpha_{1}+\alpha_{2}+\alpha_{3},\alpha_{1}+\alpha_{2}+2\alpha_{3}\}.$ Therefore, by using SES and Lemma \ref{lem 3.1},  we have $H^0(s_{3}s_{1}s_{2},\mathfrak{b})=\mathbb{C}h(\alpha_{4})\oplus  \mathfrak{g}_{-\alpha_{4}}\oplus  \mathfrak{g}_{-(\alpha_{3}+\alpha_{4})}\oplus \mathfrak{g}_{-(\alpha_{2}+\alpha_{3}+\alpha_{4})}\oplus (\mathfrak{g}_{-(\alpha_{1}+\alpha_{2}+\alpha_{3}+\alpha_{4})}\oplus \mathfrak{g}_{-(\alpha_{2}+2\alpha_{3}+\alpha_{4})})\oplus \bigoplus\limits_{ht(\beta) \ge 5}^{} \mathfrak{g}_{-\beta}$
		and $H^1(s_{3}s_{1}s_{2}, \mathfrak{b})=0.$

		Note that $\mathbb{C}h(\alpha_{4})\oplus \mathfrak{g}_{-\alpha_{4}}=V\otimes \mathbb{C}_{-\omega_{4}},$ $\langle -(\alpha_{3}+\alpha_{4}),\alpha_{4} \rangle =-1,$ $\langle -(\alpha_{2}+\alpha_{3}+\alpha_{4}),\alpha_{4} \rangle =-1,$ $\langle -(\alpha_{1}+\alpha_{2}+\alpha_{3}+\alpha_{4}),\alpha_{4} \rangle =-1,$ $\mathfrak{g}_{-(\alpha_{2}+2\alpha_{3}+\alpha_{4})}\oplus \mathfrak{g}_{-(\alpha_{2}+2\alpha_{3}+2\alpha_{4})}=V\otimes \mathbb{C}_{-\omega_{4}},$ $\mathfrak{g}_{-(\alpha_{1}+\alpha_{2}+2\alpha_{3}+\alpha_{4})}\oplus \mathfrak{g}_{-(\alpha_{1}+\alpha_{2}+2\alpha_{3}+2\alpha_{4})}=V\otimes \mathbb{C}_{-\omega_{4}},$ where $V$ is the standard two dimensional irreducible $\hat{L_{\alpha_{4}}}$-module. Further, $s_{4}$ permutes  
		$\{ \alpha_{1}+2\alpha_{2}+2\alpha_{3},\alpha_{1}+2\alpha_{2}+2\alpha_{3}+\alpha_{4},\alpha_{1}+2\alpha_{2}+2\alpha_{3}+2\alpha_{4},\alpha_{1}+2\alpha_{2}+3\alpha_{3}+\alpha_{4}, \alpha_{1}+2\alpha_{2}+3\alpha_{3}+2\alpha_{4}, \alpha_{1}+2\alpha_{2}+4\alpha_{3}+2\alpha_{4}, \alpha_{1}+3\alpha_{2}+4\alpha_{3}+2\alpha_{4}, 2\alpha_{1}+3\alpha_{2}+4\alpha_{3}+2\alpha_{4}\}.$

		Therefore, by using SES and Lemma \ref{lem 3.1}, we have

		$H^0(s_{4}s_{3}s_{1}s_{2},\mathfrak{b})= \mathfrak{g}_{-(\alpha_{1}+2\alpha_{2}+2\alpha_{3})}\oplus \mathfrak{g}_{-(\alpha_{1}+2\alpha_{2}+2\alpha_{3}+\alpha_{4})}\oplus \mathfrak{g}_{-(\alpha_{1}+2\alpha_{2}+2\alpha_{3}+2\alpha_{4})}\\\oplus\mathfrak{g}_{-(\alpha_{1}+2\alpha_{2}+3\alpha_{3}+\alpha_{4})} \bigoplus\limits_{ht(\beta) \ge 8}^{} \mathfrak{g}_{-\beta}$ 
		
		and $H^1(s_{4}s_{3}s_{1}s_{2}, \mathfrak{b})=0.$

		Note that $\langle -(\alpha_{1}+2\alpha_{2}+2\alpha_{3}), \alpha_{2} \rangle =-1,$ $\langle -(\alpha_{1}+2\alpha_{2}+2\alpha_{3}+\alpha_{4}), \alpha_{2} \rangle =-1,$ $\langle -(\alpha_{1}+2\alpha_{2}+2\alpha_{3}+2\alpha_{4}), \alpha_{2} \rangle =-1,$ $s_{2}(\alpha_{1}+2\alpha_{2}+3\alpha_{3}+\alpha_{4})= \alpha_{1}+2\alpha_{2}+3\alpha_{3}+\alpha_{4},$ and $s_{2}$ permutes all positive roots whose heights are greater than or equal to $8.$
		Therefore, by using SES and Lemma \ref{lem 3.1}, we have 
		$H^0(s_{2}s_{4}s_{3}s_{1}s_{2},\mathfrak{b})= \mathfrak{g}_{-(\alpha_{1}+2\alpha_{2}+3\alpha_{3}+\alpha_{4})} \bigoplus\limits_{ht(\beta) \ge 8}^{} \mathfrak{g}_{-\beta}$
		and $H^1(s_{2}s_{4}s_{3}s_{1}s_{2}, \mathfrak{b})=0.$

		Note that $\langle -(\alpha_{1}+2\alpha_{2}+3\alpha_{3}+\alpha_{4}),\alpha_{3} \rangle =-1,$ $\mathfrak{g}_{-(\alpha_{1}+2\alpha_{2}+3\alpha_{3}+2\alpha_{4})}\oplus \mathfrak{g}_{-(\alpha_{1}+2\alpha_{2}+4\alpha_{3}+2\alpha_{4})}=V\otimes \mathbb{C}_{-\omega_{3}},$ where $V$ is the standard two dimensional $\hat{L_{\alpha_{3}}}$-module. Further, $s_{3}$ permutes all the positive roots whose heights are greater than or equal to $10.$
		
		Therefore, by using SES and Lemma \ref{lem 3.1},  we have 
		$H^0(s_{3}s_{2}s_{4}s_{3}s_{1}s_{2},\mathfrak{b}) = \bigoplus\limits_{ht(\beta) \ge 10}^{} \mathfrak{g}_{-\beta}$
		and $H^1(s_{3}s_{2}s_{4}s_{3}s_{1}s_{2}, \mathfrak{b})=0.$
		
		Since $\langle -(\alpha_{1}+3\alpha_{2}+4\alpha_{3}+2\alpha_{4}),\alpha_{2} \rangle =-1,$ and  $\langle -(2\alpha_{1}+3\alpha_{2}+4\alpha_{3}+2\alpha_{4}),\alpha_{2} \rangle =0,$ by using SES and Lemma \ref{lem 3.1}, we have
		$H^0(s_{2}s_{3}s_{2}s_{4}s_{3}s_{1}s_{2},\mathfrak{b})=  \mathfrak{g}_{-(2\alpha_{1}+3\alpha_{2}+4\alpha_{3}+2\alpha_{4})}$ and $H^1(s_{2}s_{3}s_{2}s_{4}s_{3}s_{1}s_{2}, \mathfrak{b})=0.$
		
		Since $\langle-(2\alpha_{1}+3\alpha_{2}+4\alpha_{3}+2\alpha_{4}),\alpha_{1} \rangle =-1,$ by using SES and Lemma \ref{lem 3.1}, we have \\
		$H^0(s_{1}s_{2}s_{3}s_{2}s_{4}s_{3}s_{1}s_{2},\mathfrak{b})=0$
		and $H^1(s_{1}s_{2}s_{3}s_{2}s_{4}s_{3}s_{1}s_{2}, \mathfrak{b})=0.$
	\end{proof}

	\begin{corollary}\label{cor 7.5}We have the following
		\begin{itemize}
			\item[(i)] $H^{j}(v_{2}, \mathfrak{b})=0$ for all $j=0,1.$
			
			\item[(ii)] $H^{j}(v_{2}, \alpha_{3})=0$ for all $j=0,1.$
			
			\item[(iii)] $H^{j}(v_{2}, \mathfrak{p}_{3})=0$ for all $j=0,1.$
			
			\item[(iv)] $H^{j}(w_{i}, \mathfrak{p}_{3})=0$ for $j=0,1$ and $1\le i\le 4.$
		\end{itemize}
	\end{corollary}
	\begin{proof}
		Proof of (i): By Lemma \ref{Lemma 7.4}, we have $H^{0}(v_{2},\mathfrak{b})=0,$ and $H^1(v_{n},\mathfrak{b})=0.$ 
		
		Proof of (ii): Since $\langle \alpha_{3},\alpha_{2}\rangle=-1,$ by Lemma \ref{lem 3.1}(3), we have $H^j(v_{2},\alpha_{3})=0$ for $j=0,1.$

		Proof of (iii):
		Consider the exact sequence
		\begin{center}
			$0\longrightarrow \mathfrak{b}\longrightarrow\mathfrak{p}_{3}\longrightarrow\mathbb{C}_{\alpha_{3}}\longrightarrow0$
		\end{center}
		of $B$-modules.
		
		Then we have the following long exact sequence 
		\begin{center}
			$0\longrightarrow H^0(v_{2},\mathfrak{b})\longrightarrow H^0(v_{2},\mathfrak{p}_{3})\longrightarrow H^0(v_{2},\alpha_{3})\longrightarrow$ 	$H^1(v_{2}, \mathfrak{b})\longrightarrow H^1(v_{2},\mathfrak{p}_{3})\longrightarrow H^1(v_{2},\alpha_{3})\longrightarrow H^2(v_{2},\mathfrak{b})\longrightarrow\cdots$
		\end{center}
		of $B$-modules.
		
		Therefore, by using $(i)$ and $(ii),$ proof of $(iii)$ follows.

		Proof of (iv): Note that $w_{i}=x_{i}v_{2}$ and $\ell(w_{i})=\ell(x_{i})+\ell(v_{2})$ for all $1\le i\le 4.$ Therefore, by (iii) and using SES proof of (iv) follows. 
	\end{proof}
	
	\begin{proposition}\label{prop 7.6}
		We have $w_{i}\in W^{P_{3}}$ and $P_{i}=Aut^0(X_{P_{3}}(w_{i}))$ for all $1\le i\le 4.$
	\end{proposition}
	\begin{proof}
		By Lemma \ref{lemma 7.2}(iii),  $w_{i}(\alpha_{3})$ is a non-simple positive root for all $1\le i\le 4.$ In particular, we have $w_{i}\in W^{\alpha_{3}}$ for all $1\le i\le 4.$  On the other hand, by Lemma \ref{lemma 7.3}, we have the following:
		\begin{itemize}
			\item[(i)] $w_{i}^{-1}(\alpha_{i})$ is a negative root for all $1\le i\le 4.$\\
			\item[(ii)] $w_{i}^{-1}(\alpha_{j})$ is a positive root for all $1\le i,j\le 4$ such that $j\neq i.$
		\end{itemize}

		Thus, $P_{i}$ is the stabilizer of $X_{P_{3}}(w_{i})$ in $G$ for all $1\le i\le 4.$ Since $\alpha_{0}=\omega_{1},$ and  $v_{2}^{-1}(\alpha_{0})$ is a negative root, $w_{i}^{-1}(\alpha_{0})=v_{2}^{-1}(\alpha_{0})$ (as$~w_{0,S\setminus\{\alpha_{1}\}}w_{0,S\setminus\{\alpha_{1}, \alpha_{i}\}}(\alpha_{0})=\alpha_{0}$) is a negative root. Therefore, by using \cite[Theorem 6.6, page 781]{Kan} the natural homomorphism $\varphi: P_{i}\longrightarrow Aut^0(X_{P_{3}}(w_{i}))$ is an injective homomorphism of algebraic groups.
		
		Now consider the following short exact sequence 
		\begin{center}
			$0\longrightarrow \mathfrak{p}_{3}\longrightarrow \mathfrak{g}\longrightarrow \mathfrak{g/p}_{3}\longrightarrow 0$	
		\end{center}
		of $B$-modules.
		
		Consider the following long exact sequence of $B$-modules induced by the above short exact sequence of $B$-modules 
		\begin{center}
			$0\longrightarrow H^0(w_{i},\mathfrak{p}_{3})\longrightarrow H^0(w_{i},\mathfrak{g})\longrightarrow H^0(w_{i}, \mathfrak{g/p}_{3})\longrightarrow$ 
			
			$H^1(w_{i},\mathfrak{p}_{3})\longrightarrow H^1(w_{i},\mathfrak{g})\longrightarrow H^1(w_{i}, \mathfrak{g/p}_{3})\longrightarrow\cdots $	
		\end{center}

		Since $H^0(w_{i},\mathfrak{g})=\mathfrak{g},$ by  using Corollary \ref{cor 7.5}(iv),  we have $H^0(w_{i},\mathfrak{g/p}_{3})=\mathfrak{g}.$ Further, since $d\varphi(\mathfrak{p}_{i})\subseteq$Lie$(Aut^0(X_{P_{3}}(w_{i})))\subseteq H^0(w_{i},\mathfrak{g/p}_{3})=\mathfrak{g}$ and the base field  is $\mathbb{C},$ $Aut^0(X_{P_{3}}(w_{i}))$ is a closed subgroup of $G$ and in fact, it is a stabilizer of $X_{P_{3}}(w_{i})$ in $G$ (see \cite[Theorem 12.5, p.85 and Theorem 13.1,p.87]{Hum2}) Thus, we have  $Aut^0(X_{P_{3}}(w_{i}))=P_{i}.$ 
	\end{proof}
	
	\section{$G$ is of type $G_{2}$}
	In this section, we assume that $G$ is of type $G_{2}.$ Further, we prove that for any $1\le i\le 2,$ there exists a Schubert variety $X_{P_{1}}(w_{i})$ in $G/P_{1}$ such that $P_{i}=Aut^0(X_{P_{1}}(w_{i})).$ 
	
	Since $G$ is of type $G_{2},$ we have  $\alpha_{0}=3\alpha_{1}+2\alpha_{2}=\omega_{2}.$ Then $v_{2}=s_{2}s_{1}s_{2}$ is the unique element of minimal length in $W$ such that $v_{2}^{-1}(\alpha_{0})=-\alpha_{2}$ as in Corollary \ref{cor 2.2}. Then $v_{2}$ satisfies the following
	\begin{lemma}\label{lemma 8.1} Then we have
		\begin{itemize}
			\item[(i)] $v_{2}^{-1}(\alpha_{0})$ is a negative root. 
			
			\item[(ii)] $v_{2}(\alpha_{1})=2\alpha_{1}+\alpha_{2}.$
			
			\item[(iii)] $(s_{1}v_{2})^{-1}(\alpha_{0})$ is a negative root. 
			
			\item[(iv)] $s_{1}v_{2}(\alpha_{1})=\alpha_{1}+\alpha_{2}.$ 
		\end{itemize}
	\end{lemma}
	\begin{proof}
		Proof of $(i),(ii),(iii)$ and $(iv)$ follows from the usual calculation. 
	\end{proof}
	Let $w_{1}=s_{1}v_{2}$ and $w_{2}=v_{2}.$
	\begin{lemma}\label{lemma 8.2} Then we have
		\begin{itemize}
			\item [(i)] $H^i(w_{2}, \alpha_{1})=0$ for $i=0,1.$
			
			\item [(ii)] $H^i(w_{2}, \mathfrak{b})=0$ for $i=0,1$
			
			\item [(iii)] $H^i(w_{2}, \mathfrak{p}_{1})=0$ for $i=0,1$
			
			\item [(iv)] $H^i(w_{1}, \mathfrak{p}_{1})=0$
			for $i=0,1.$	
		\end{itemize}
	\end{lemma}
	\begin{proof}
		Proof of (i): Since $\langle \alpha_{1}, \alpha_{2} \rangle=-1,$ by Lemma \ref {lem 3.1} we  have $H^i(w_{2},\alpha_{1})=0$ for $i=0,1.$
		
		Proof of (ii): Note that $\mathfrak{b}=\mathbb{C}h(\alpha_{1})\oplus \mathbb{C}h(\alpha_{2})\oplus \mathbb{C}_{-\alpha_{1}}\oplus \mathbb{C}_{-\alpha_{2}}\oplus \mathbb{C}_{-(\alpha_{1}+\alpha_{2})}\oplus \mathbb{C}_{-(2\alpha_{1}+\alpha_{2})}\oplus \mathbb{C}_{-(3\alpha_{1}+\alpha_{2})}\oplus \mathbb{C}_{-(3\alpha_{1}+2\alpha_{2})}.$
		
		Since $\mathbb{C}h(\alpha_{2})\oplus \mathbb{C}_{-\alpha_{2}}=V\otimes \mathbb{C}_{-\omega_{2}},$ where $V$ is the standard two dimensional $\hat{L}_{\alpha_{2}}$-module, by using Lemma \ref{lem 3.1} we have $H^i(s_{2}, \mathbb{C}h(\alpha_{2})\oplus \mathbb{C}_{-\alpha_{2}})=0$ for $i=0,1.$ Note that $s_{2}$ permutes all the positive roots other than $\alpha_{2}.$
		
		Therefore, we have $H^0(s_{2}, \mathfrak{b})=\mathbb{C}h(\alpha_{1})\oplus \mathbb{C}_{-\alpha_{1}}\oplus \mathbb{C}_{-(\alpha_{1}+\alpha_{2})}\oplus \mathbb{C}_{-(\alpha_{2}+2\alpha_{1})}\oplus \mathbb{C}_{-(3\alpha_{1}+2\alpha_{2})}\oplus \mathbb{C}_{-(3\alpha_{1}+\alpha_{2})}$ and $H^1(s_{2}, \mathfrak{b})=0.$
		
		Note that $\mathbb{C}h(\alpha_{1})\oplus \mathbb{C}_{-\alpha_{1}}=V\otimes \mathbb{C}_{-\omega_{1}},$ where $V$ is the standard two dimensional $\hat{L_{\alpha_{1}}}$-module, and $\mathbb{C}_{-(\alpha_{1}+\alpha_{2})}\oplus \mathbb{C}_{-(2\alpha_{1}+\alpha_{2})}\oplus \mathbb{C}_{-(3\alpha_{1}+\alpha_{2})}=V^{'}\otimes\mathbb{C}_{-\omega_{1}}$ where $V^{'}$ is the standard three dimensional $\hat{L}_{\alpha_{1}}$ module. 
		
		Therefore, by using Lemma \ref{lem 3.1} we have $H^i(s_{1},(\mathbb{C}h(\alpha_{1})\oplus \mathbb{C}_{-\alpha_{1}})\oplus (\mathbb{C}_{-(\alpha_{1}+\alpha_{2})}\oplus \mathbb{C}_{-(2\alpha_{1}+\alpha_{2})}\oplus \mathbb{C}_{-(3\alpha_{1}+\alpha_{2})}))=0$ for $i=0,1.$ 
		
		Since $\langle -(3\alpha_{1}+2\alpha_{2}), \alpha_{1}\rangle =0,$ by using Lemma \ref {lem 3.1} we have $H^0(s_{1},\mathbb{C}_{-(3\alpha_{1}+2\alpha_{2})})=\mathbb{C}_{-(3\alpha_{1}+2\alpha_{2})}$ and $H^1(s_{1},\mathbb{C}_{-(3\alpha_{1}+2\alpha_{2})})=0.$ 
		
		Therefore, combining the above discussion we have $H^1(s_{1},H^0(s_{2},\mathfrak{b}))=0$ and $H^0(s_{1}s_{2},\mathfrak{b})=\mathbb{C}_{-(3\alpha_{1}+2\alpha_{2})}.$ Since $H^0(s_{1}, H^1(s_{2},\mathfrak{b}))=0,$ and $H^1(s_{1},  H^0(s_{2},\mathfrak{b}))=0,$ by using SES we have $H^1(s_{1}s_{2},\mathfrak{b})=0.$
		Since $\langle -(3\alpha_{1}+2\alpha_{2}), \alpha_{2}\rangle =-1,$ by using Lemma \ref {lem 3.1} we have  $H^0(w_{2}, \mathfrak{b})=H^0(s_{2}, H^0(s_{1}s_{2}, \mathfrak{b}))=H^0(s_{2}, \mathbb{C}_{-(3\alpha_{1}+2\alpha_{2})}))=0,$ and $H^1(s_{2}, H^0(s_{1}s_{2},\mathfrak{b}))=0.$
		Since $H^1(s_{1}s_{2}, \mathfrak{b})=0,$ we have $H^0(s_{2}, H^1(s_{1}s_{2}, \mathfrak{b}))=0.$ Therefore, by using SES we have $H^1(w_{2}, \mathfrak{b})=0.$
		
		Proof of (iii): Consider the short exact sequence 
		\begin{center}
			$0\longrightarrow \mathfrak{b}\longrightarrow\mathfrak{p}_{1}\longrightarrow\mathbb{C}_{\alpha_{1}}\longrightarrow0$
		\end{center}
		of $B$-modules.
		
		Then we have the following long exact sequence 
		\begin{center}
			$0\longrightarrow H^0(w_{2},\mathfrak{b})\longrightarrow H^0(w_{2},\mathfrak{p}_{1})\longrightarrow H^0(w_{2},\alpha_{1})\longrightarrow$ 	$H^1(w_{2}, \mathfrak{b})\longrightarrow H^1(w_{2},\mathfrak{p}_{1})\longrightarrow H^1(w_{2},\alpha_{1})\longrightarrow H^2(w_{2},\mathfrak{b})\longrightarrow\cdots$
		\end{center}
		of $B$-modules.
		
		Therefore, by using $(i)$ and $(ii),$ proof of $(iii)$ follows. 
		
		Proof of (iv): By $(iii)$ we have $H^i(w_{2}, \mathfrak{p}_{1})=0$ for $i=0,1.$ Further, note that $w_{1}=s_{1}w_{2}$ and $\ell(w_{1})=\ell(w_{2})+1.$  Therefore, by using SES, we have $H^i(w_{1}, \mathfrak{p}_{1})=0$ for $i=0,1.$
	\end{proof}

	Observe that by Lemma \ref{lemma 8.1}(ii),(iv), we have $w_{i} \in W^{P_{1}}.$
	\begin{proposition}\label{prop 8.3}
		We have $P_{i}=Aut^0(X_{P_{1}}(w_{i}))$ for $i=1,2.$
	\end{proposition}
	\begin{proof}
		By Lemma \ref{lemma 8.1}(ii),(iv), $w_{i}(\alpha_{1})$ is a non-simple positive root. On the other hand, we have the following:
		
		\begin{itemize}
			\item[(i)]  $w_{i}^{-1}(\alpha_{i})$ is a negative root.\\
			
			\item[(ii)]  $w_{i}^{-1}(\alpha_{j})$ is a positive root for $j\neq i.$
		\end{itemize}

		Thus $P_{i}$ is the stabilizer of $X_{P_{1}}(w_{i})$ in $G.$ Since $\alpha_{0}=\omega_{2},$ and  $v_{2}^{-1}(\alpha_{0})$ is a negative root, $w_{i}^{-1}(\alpha_{0})=v_{2}^{-1}(\alpha_{0})$ is a negative root. Therefore, by using \cite[Theorem 6.6, p.781]{Kan} the natural homomorphism $\varphi :P_{i}\longrightarrow Aut^0(X_{P_{1}}(w_{i}))$ is an injective homomorphism of algebraic groups. 
		
		Now consider the following SES 
		\begin{center}
			$0\longrightarrow \mathfrak{p}_{1}\longrightarrow \mathfrak{g}\longrightarrow \mathfrak{g/p}_{1}\longrightarrow 0$	
		\end{center}
		of $B$-modules.
		
		Thus, we have the following long exact sequence 
		\begin{center}
			$0\longrightarrow H^0(w_{i},\mathfrak{p}_{1})\longrightarrow H^0(w_{i},\mathfrak{g})\longrightarrow H^0(w_{i}, \mathfrak{g/p}_{1})\longrightarrow$ 
			
			$H^1(w_{i},\mathfrak{p}_{1})\longrightarrow H^1(w_{i},\mathfrak{g})\longrightarrow H^1(w_{i}, \mathfrak{g/p}_{1})\longrightarrow\cdots $	
		\end{center}
		of $B$-modules.  
		
		Since $H^0(w_{i},\mathfrak{g})=\mathfrak{g},$ by  using Lemma \ref{lemma 8.2}(iii),(iv),  we have $H^0(w_{i},\mathfrak{g/p}_{1})=\mathfrak{g}.$ Further, since $d\varphi(\mathfrak{p}_{i})\subseteq$Lie$(Aut^0(X_{P_{1}}(w_{i})))\subseteq H^0(w_{i},\mathfrak{g/p}_{1})=\mathfrak{g}$ and the base field  is $\mathbb{C},$ $Aut^0(X_{P_{1}}(w_{i}))$ is a closed subgroup of $G$ containing $P_{i}$(see \cite[Theorem 12.5, p.85 and Theorem 13.1,p.87]{Hum2}). Hence, $Aut^0(X_{P_{1}}(w_{i}))$ is the stabilizer of $X_{P_{1}}(w_{i})$ in $G.$ Thus, we have  $Aut^0(X_{P_{1}}(w_{i}))=P_{i}.$  	
	\end{proof}
	
	\section{Main Theorem}
	Let $G$ be a simple algebraic group of adjoint type over $\mathbb{C}.$ In this section, we prove that $\alpha_{r}$ is co-minuscule if and only if for any parabolic subgroup $Q$ containing $B$ properly, there is no Schubert variety $X_{Q}(w)$ in $G/Q$ such that $P_{r}=Aut^{0}(X_{Q}(w)).$

	Recall that given any parabolic subgroup $P$ of $G$ containing $B$ properly there exists a Schubert variety $X(w)$ in $G/B$ such that $P=Aut^0(X(w))$ (see \cite[Theorem 2.1,p.543]{KP1}). The existence of such a Schubert variety in a partial flag variety may not be true. This proposition illustrates this fact.
	\begin{proposition}\label{theorem 9.1} Let $\alpha_{r}$ be co-minuscule. If there exists a Schubert variety $X_{Q}(w)$ in  $G/Q$ for some parabolic subgroup $Q$ of $G$ containing $B$ such that $P_{r}=Aut^0(X_{Q}(w)),$ then we have $Q=B.$
	\end{proposition}
	\begin{proof}
		First note that if $G$ is simply-laced, then $\omega_{i}$ is minuscule if and only if $\alpha_{i}$ is co-minuscule (see \cite[Lemma 3.1]{KS2}). So, the proposition is true when $G$ is simply-laced (see \cite[Theorem 9.1]{KS2}). So, we may assume that $G$ is non-simply-laced.
		
		Assume that $P_{r}=Aut^0(X_{Q}(w))$  for some parabolic subgroup $Q$ containing $B$ properly and for some $w\in W^{Q}.$ Then there exists a non-empty subset $J$ of $S$ such that $Q=P_{J}.$ Since $P_{r}=Aut^0(X_{Q}(w)),$ we have $w^{-1}(\alpha_{j})>0$ for all $j\neq r,$ i.e., we have $w^{-1}\in W^{S \setminus \{\alpha_{r}\}}.$ Since the action $P_{r}$ on $X_{Q}(w)$ is faithful, by \cite[Theorem 6.6, p.781]{Kan} we have $w^{-1}(\alpha_{0})<0.$ Further, since $w^{-1}(\alpha_{0})<0$ and $w^{-1}\in W^{S \setminus \{\alpha_{r}\}},$ by Lemma \ref{lemma 1.2} we have $w^{-1}=w_{0}^{S \setminus \{\alpha_{r}\}}.$ Thus, we have $w=({w_{0}^{S \setminus \{\alpha_{r}\}}})^{-1}.$ Note that $(w_{0}^{S\setminus\{\alpha_{r}\}})^{-1}=w_{0,S\setminus\{\alpha_{r}\}}w_{0}=w_{0}(w_{0}w_{0,S\setminus\{\alpha_{r}\}}w_{0})=w_{0}w_{0,S\setminus\{\alpha_{r}\}},$ as $w_{0}(\alpha_{i})=-\alpha_{i}$ for all $1\le i\le n$ (see \cite[p.216, p.217, p.233]{Bou}). Therefore, we have $w={w_{0}^{S \setminus \{\alpha_{r}\}}}.$ Hence, we have $J\subset S\setminus \{\alpha_{r}\}.$ Let $\tau$ be the Dynkin diagram automorphism of $S\setminus \{\alpha_{r}\}$ induced by  $-w_{0, S \setminus \{\alpha_{r}\}}.$ Now since $J$ is nonempty, there exists $\alpha_{i} \in J$ such that $\alpha_{i}\neq \alpha_{r}.$ So, there exists $\alpha_{j}\in S\setminus \{\alpha_{r}\}$ such that $\tau(\alpha_{j})=\alpha_{i}.$ Now we have $w^{-1}s_{j}w=w_{0}w_{0, S\setminus \{\alpha_{r}\}}s_{j}w_{0, S\setminus \{\alpha_{r}\}}w_{0}=s_{\tau(\alpha_{j})}=s_{i}.$ Therefore, we have $s_{j}w=w~ mod~W_{Q}.$ This is a contradiction to the fact that $P_{r}$ is the stabilizer of $X_{Q}(w)$ in $G.$ Hence, we have $Q=B.$
	\end{proof}
	\begin{theorem}\label{theorem 9.2}
		$\alpha_{r}$ is  co-minuscule if and only if for any parabolic subgroup $Q$ containing $B$ properly, there is no Schubert variety $X_{Q}(w)$ in $G/Q$ such that $P_{r}=Aut^0(X_{Q}(w)).$
	\end{theorem}
	\begin{proof}
		First note that if $G$ is simply-laced, then $\omega_{i}$ is minuscule if and only if $\alpha_{i}$ is co-minuscule (see \cite[Lemma 3.1]{KS2}). So, the theorem is true when $G$ is simply-laced (see \cite[Theorem 9.2]{KS2}). So, we may assume that $G$ is non-simply-laced.
		
		Now proof of the theorem follows from Proposition \ref{prop5.8}, Proposition \ref{prop6.7}, Proposition \ref{prop 7.6}, Proposition \ref{prop 8.3}, and Proposition \ref{theorem 9.1}.
	\end{proof}
	
	We conclude this article by the following problem.
	
	{\bf Problem :}
	There is an interesting question arise out of this work. Let $G$ be a simple algebraic group of adjoint type. What are the parabolic subgroups $P$ of $G$ containing a minimal parabolic subgroup properly, such that  $P=Aut^0(X(w))$ for some Schubert variety in a partial flag variety?

	{\bf Acknowledgements:}
	The first named author thanks the Infosys Foundation for the partial financial support. He also thanks MATRICS for the partial
	financial support.

\end{document}